\documentclass{amsart}

\usepackage{graphicx}
\usepackage{color}

\usepackage{framed}
\usepackage{amsthm}
\usepackage{amscd}
\usepackage{amsmath}
\usepackage{mathrsfs}
\usepackage{amsfonts}
\usepackage{amssymb}
\usepackage{graphics}
\usepackage{amstext}
\usepackage{constants} 
\usepackage[colorlinks]{hyperref}

\hypersetup{ colorlinks=true, pdftitle={Shy shadows of infinite-dimensional partially hyperbolic invariant sets}, pdfsubject={Shy shadows of infinite-dimensional partially hyperbolic invariant sets}, pdfauthor={Daniel Smania},pdfkeywords={  null set, shy set, prevalence, prevalent, haar null set, horseshoe,  infinite dimension,  invariant set, infinite-dimensional dynamical system, partially hyperbolic, hyperbolic, compact map, Banach space, invariant cones, dense image, dominated splitting}}

\newconstantfamily{c}{
symbol=\lambda,
format=\arabic,
}
\newconstantfamily{e}{
symbol=\theta,
format=\arabic,
}

\begin{document}

\author{Daniel Smania }

\address{Departamento de Matem\'atica \\
   ICMC/USP - S\~ao Carlos \\
               Caixa Postal 668 \\ S\~ao Carlos-SP \\ CEP 13560-970 \\ Brazil.}
\email{smania@icmc.usp.br} \urladdr{\href{http://www.icmc.usp.br/pessoas/smania/}{www.icmc.usp.br/pessoas/smania/}}

\date{\today}
\title[Shy shadows of $\infty$-dim. partially hyperbolic invariant sets]{ 
Shy shadows of infinite-dimensional partially hyperbolic invariant sets}

\begin{abstract}   Let $\mathcal{R}$ be a  strongly compact $C^2$ map defined in an open subset of an infinite-dimensional  Banach space  such that  the image of its derivative $D_F \mathcal{R}$ is dense for  every $F$. Let $\Omega$ be a compact,  forward invariant and partially hyperbolic set of $\mathcal{R}$  such that $\mathcal{R}\colon \Omega\rightarrow \Omega$ is onto. The $\delta$-shadow $W^s_\delta(\Omega)$ of $\Omega$ is the union of the sets 
$$W^s_\delta(G)= \{F\colon dist(\mathcal{R}^iF,  \mathcal{R}^iG) \leq \delta, \ for \ every \  i\geq 0  \},$$
where $G \in \Omega$.  Suppose that $W^s_\delta(\Omega)$ has transversal empty interior, that is,   for every   $C^{1+Lip}$ $n$-dimensional  manifold $M$  transversal   to the distribution of dominated directions of $\Omega$  and sufficiently close to $W^s_\delta(\Omega)$ we have that $M\cap W^s_\delta(\Omega)$ has empty interior in $M$. Here $n$ is the finite dimension of the strong unstable direction. We show that  if $\delta'$ is small enough then $$\cup_{i\geq 0}\mathcal{R}^{-i}W^s_{\delta'} (\Omega)$$ intercepts a  $C^k$-generic  finite dimensional curve inside the Banach space in a set of parameters with zero Lebesgue measure, for every $k\geq 0$.  This extends to infinite-dimensional dynamical systems previous  studies  on  the Lebesgue measure of stable laminations of invariants sets.  \end{abstract}

\subjclass[2010]{37D30,37D20,  37L45, 46T20,   37D25} \keywords{  null set, shy set, prevalence, prevalent, haar null set, horseshoe,  infinite dimension,  invariant set, infinite-dimensional dynamical system, partially hyperbolic, hyperbolic, compact map, Banach space, invariant cones, dense image, dominated splitting}

\thanks{D.S. was partially supported by CNPq 470957/2006-9, 310964/2006-7,  472316/03-6,  303669/2009-8, 305537/2012-1, 307617/2016-5 and FAPESP 03/03107-9,  2008/02841-4, 2010/08654-1.}

\maketitle
\newcommand{\co}{\mathbb{C}}
\newcommand{\incl}[1]{i_{U_{#1}-Q_{#1},V_{#1}-P_{#1}}}
\newcommand{\inclu}[1]{i_{V_{#1}-P_{#1},\co-P}}
\newcommand{\func}[3]{#1\colon #2 \rightarrow #3}
\newcommand{\norm}[1]{\left\lVert#1\right\rVert}
\newcommand{\norma}[2]{\left\lVert#1\right\rVert_{#2}}
\newcommand{\hiper}[3]{\left \lVert#1\right\rVert_{#2,#3}}
\newcommand{\hip}[2]{\left \lVert#1\right\rVert_{U_{#2} - Q_{#2},V_{#2} -
P_{#2}}}
\newtheorem{prop}{Proposition}[section]

\newtheorem{lem}[prop]{Lemma}

\newtheorem{rem}[prop]{Remark}

\newtheorem*{mth}{Main Theorem}
\newtheorem{thm}{Theorem}

\newtheorem{cor}[prop]{Corollary}

\setcounter{tocdepth}{1}
\tableofcontents

\section{Introduction.}

In many  areas of Mathematics, when we need to classify/study  objects in a large family, it is quite often the case that one can not understand {\it all} objects in this family, but just {\it most}  of them.  If this family of objects  is an open subset of  a Banach space (or a Banach manifold), as for instance in the study of the typical behaviour of a large family of smooth discrete dynamical systems,  the meaning of  ``most''   it is not at all obvious, since   infinite dimensional spaces do not carry  a natural class of regular measures (in  finite dimensional linear spaces  the  class of Lebesgue measures  is  such class). Sometimes by ``most''  one mean either an open and dense set or a residual set. But even in the finite dimensional case those are not equivalent with  full Lebesgue measure sets.  Moreover  it may happens   in  the study of dynamical systems that the topological generic behaviour does not coincide with the typical measure-theoretical   behaviour. See  Hunt, Sauer and Yorke \cite{hunt} and Hunt and Kaloshin \cite{hk}   for a large number of examples.

In dynamical systems  it came  into  favor the idea suggested by   Kolmogorov \cite{kolmo}  that  it is good enough  to understand the dynamical behaviour  at {\it  almost every parameter}  in a smooth finite-dimensional family of dynamical systems that belongs  to a {\it residual set of families}. This approach was very  successful, specially in one-dimensional dynamics. See for instance the  Palis'  conjectures  \cite{palis} \cite{palis2},  the results in Avila, Lyubich and de Melo \cite{alm} and  Avila and Moreira \cite{am} and the  survey by Hunt and Kaloshin \cite{hk} on prevalence. Kaloshin \cite{k}(see also \cite{hk}) defined a {\it nonlinear prevalence}  on $C^r(M,N)$, where $M$ and $N$ are manifolds, in the spirit of  Kolmogorov's suggestion. 

It turns out that to understand the behaviour of most one-dimensional dynamical systems it is  often necessary  (Avila, Lyubich and de Melo \cite{alm}) to understand the dynamics of a highly non-linear, complex analytic compact operator acting on a Banach space of complex analytic dynamical systems, the  {\it renormalization operator}. See Lyubich  \cite{lyu} and  de Faria, de Melo and Pinto\cite{fma} for the unimodal case, and   \cite{sm} for the definition of renormalization in the multimodal case. The action of this operator acting on unimodal maps is hyperbolic on  its omega limit set  \cite{lyu} \cite{fma}. The stable lamination of this omega limit set consists of  all infinitely renormalizable maps. A family  of dynamical systems is a   finite-dimensional smooth curve inside this Banach space, so if we want to know how large is the set of  parameters in this family that corresponds to infinitely renormalizable maps, we need to know how such  curve intercepts the stable lamination. So the typical  behaviour in the {\it parameter } space of  one-dimensional  multimodal maps is closely connected with the typical behaviour  in the infinite-dimensional  {\it phase} space of  the renormalization operator. 

In the case of renormalization theory of unimodal maps, it was proven by  Avila, Lyubich and de Melo \cite{alm} that for  typical $1$-dimensional  real-analytic family of unimodal maps the set of parameters corresponding to infinitely renormalizable maps  has zero Lebesgue measure. An essential step of the proof is to show  that   the stable lamination extends to a codimension-one complex analytic  lamination of the whole space of of maps (except maybe for a few maps with very  specific combinatorics)  and as a consequence the holonomy of this lamination  is a quasisymmetric map. This quite special regularity  can be exploited to conclude the result.  Indeed, using that  the stable lamination has codimension  one  de Faria, de Melo and Pinto\cite{fma} proved that the holonomy of the stable lamination  is  $C^{1+\epsilon}$. The renormalization operator  of multimodal maps has a   stable lamination with  higher codimension and such approach does not seem to be applicable in this case, so we need new tools.  

Questions about  the measure of invariant laminations started with  the work of Bowen  \cite{bowen} (see also Bowen and Ruelle \cite{br} for a similar  result for flows), that proved  that the stable lamination of  a basic hyperbolic set of a $C^2$ diffeomorphism in a finite-dimensional manifold  has zero Lebesgue measure  if and only if it has empty interior. That result was generalized by Alves and Pinheiro \cite{ap} for horseshoe-like partially hyperbolic invariant sets. To the best of our knowledge, there is no result of this kind in the literature for infinite dimensional dynamical
systems. The literature for measure-theoretical/ergodic  theory of infinite-dimensional maps as partially hyperbolic maps considered here is indeed   scarce. We should mention the  work of Ma\~ne \cite{mane}, where he extends  Oseledec theorem to  compact smooth maps acting on Banach spaces   and the works of  Lian and Young \cite{ly1} \cite{ly2} and Lian, Young and Zeng \cite{lyz} on infinite-dimensional dynamical systems.

We will show that for certain  compact  partially hyperbolic set  on infinite dimensional  Banach spaces a generic curve intercepts  the  stable lamination in a  set of zero Lebesgue measure. In \cite{smania2} we use this result to show that in a generic finite-dimensional family of real-analytic multimodal maps the set of parameters of maps that are infinitely renormalizable with bounded combinatorics has zero Lebesgue measure.  However this   does not exhaust the applications of our results. Indeed, one can expect their application in other fields, as the study of dissipative  PDEs.

We  adopted an elementary approach, both in methods and statements of the results. In particular  we did  not need to prove (or to use previous results on) the existence of invariant laminations for  a partially hyperbolic invariant set. 

There are many   difficulties  dealing with the general Banach space setting. Firstly, we need to conceive a  notion of ``thin'' set that could be suitable and useful in this setting, since  the notion of ``zero Lebesgue measure'' does not make sense anymore. There are many notions for  ``thin'' sets in Banach spaces.  Haar null sets (shy sets) were introduced by Christensen \cite{haarnull} and  reintroduced by Hunt, Sauer and Yorke \cite{hunt}.   J. Lindenstrauss and D. Preiss \cite{lind}  introduced the  $\Gamma$-null sets in the study of the generalization to the Banach space setting of Rademacher's  Theorem on the almost everywhere differentiability of Lipchitz functions.     There are also stronger notions of null sets, as cubic, Gaussian and Aronszajn null sets. Those stronger notions coincide in separable Banach spaces, see Cs{\"o}rnyei \cite{coincide}. We are going to  introduce  the concept of $\Gamma^k$-null set (see Section \ref{nullsection}), that somehow unifies  Kolmogorov's  idea and  Lindenstrauss \& Preiss's  approaches in the abstract setting of Banach spaces rather than  $C^r(M,N)$. Indeed even when $C^r(M,N)$ is a Banach space  the  set of smooth curves  in $C^r(M,N)$ considered by  Kaloshin, as well the topology on this set,  are  somehow different of  those used here. 

Secondly, we need of a ``strong compactness''  assumption (see Section \ref{strong_c}) to prove the pre-compactness of a sequence of  curves transversal to the central stable direction  obtained iterating our dynamical system (and cutting) on a given transversal curve.   This is similar to the   graph method approach for the proof of the Stable Manifold Theorem and the $\lambda$-Lemma.  The lack of contraction in the center-stable direction  does not allow us to prove the contraction of this "graph map"-like process, however the pre-compactness will be enough to our purposes. Such pre-compactness is obtained in Section \ref{action_t} with  careful estimates using both ``strong'' $|\cdot|_1$  and ``weak'' $|\cdot|_0$ norms that appears in the ``strong compactness''  assumption.  Such a difficulty does not appear in the previous finite dimensional results, since  in this setting every $C^2$ map is obviously strongly compact. 

Finally, as usual in smooth ergodic theory, it is essential to be able to control how measures on finite-dimensional curves are perturbed by  the iteration of the dynamics. This is sometimes called ``bounded distortion control'' and in the finite-dimensional setting the jacobian and the classical formulas for  change of variables of integrals plays a crucial role to obtain these estimates. In the Banach space setting we need to consider Lebesgue measures on finite-dimensional smooth curves inside the Banach space. There is not a canonical way to do this. One may prefer either the  full-dimensional Haussdorf measure induced by  the norm of the Banach space restricted to such curve in certain situations, that has the advantage of being  coordinate-free,  or a Lebesgue measure induced by some ad-hoc inner product on the (finite-dimensional) parameter space of the curve, that could be more handy for  bounded distortion estimates. All these measures will play a role in the proof of the main result and we need to be careful when we change the measures on consideration along our arguments. This is done in Section \ref{lower_bound} and Section \ref{dballs}.

\section{Statement of the main results.} \label{strong_c}

Let $\mathcal{B}_0$ be a  Banach space, either real or complex, with norm $|\cdot|_0$.  We say that a  $C^2$ {F}r\'echet  differentiable  map  $$\mathcal{R}\colon V \rightarrow \mathcal{B}_0,$$
where $V \subset \mathcal{B}_0$ is an open set, is a {\bf strongly compact  $C^2$  map} if  
\begin{itemize}
\item[A.] There is a subspace $\mathcal{B}_1\subset \mathcal{B}_0$, that endowed with a norm $|\cdot|_{1}$ is a Banach space,  such that  the inclusion 
$$i\colon (\mathcal{B}_1, |\cdot|_1) \mapsto (\mathcal{B}_0, |\cdot|_0)$$
is a compact linear map, and 
\item[B.] We can write 
$$\mathcal{R}= i\circ \tilde{\mathcal{R}},$$
where $\tilde{\mathcal{R}}$ is a  $C^2$ {F}r\'echet  differentiable  map  
$$\tilde{\mathcal{R}} \colon (V, |\cdot|_0) \rightarrow (\mathcal{B}_1, |\cdot|_1).$$
\end{itemize}
Note that if $\mathcal{B}_0$ has infinite dimension then $\mathcal{R}$  is not  a diffeomorphism. 
 
We will say that a forward $\mathcal{R}$-invariant compact subset  $\Omega \subset \mathcal{B}_1$ is a {\bf partially hyperbolic set}  of $\mathcal{R}$ if there are  two continuous $\mathcal{R}$-invariant distributions of subspaces of $\mathcal{B}_0$, the {\bf horizontal directions}
$$G \in \Omega \mapsto E^{h}_G,$$
 and  the  {\bf unstable directions } 
$$G \in \Omega\mapsto E^u_G,$$
such that $\mathcal{B}_0= E^h_G \oplus E^u_G$  and moreover 
\begin{itemize}
\item[A1.] There exist $n_0$,  $C > 0$, $\delta_1> 0$ and $\lambda_G > 0$ such that if  $ v\in E^{h}_G$ then  
$$|D_F\mathcal{R}^{n_0}\cdot v|_{0}\leq \lambda_G  |v|_{0}.$$
\item[A2.] There exists $\theta > 1$ and  for every $G \in \Omega$  there exists $\theta_G > \theta \max \{ 1,  \lambda_G \}$ 
$$|D_F\mathcal{R}^{n_0}\cdot v|_{0}\geq  \theta_G  |v|_{0}. $$ 
\end{itemize}

For each $G \in \Omega$, define the $\delta$-shadow of $G$ as
$$W^s_\delta(G)= \{F \in V \colon \ dist_{\mathcal{B}_0}(\mathcal{R}^k F,  \mathcal{R}^k G)\leq \delta, \ for \ every \ k\geq 0 \}.$$ 
and the $\delta$-shadow of $\Omega$ as
$$W^s_\delta(\Omega)=\cup_{G \in \Omega} W^s_\delta(G).$$
It is easy to see that $W^s_\delta(\Omega)$ is a closed subset of $V$ (in $\mathcal{B}_0$ topology). 

Since $\mathcal{R}$ is a compact map we have that $E^u_F$ has finite dimension. Without loss of generality we are going to assume that $n= \dim_{\mathbb{R}} E^u_G$ does not depend on $G \in \Omega$. We say that $W^s_\delta(\Omega)$ has {\bf real transversal empty interior},  if   there is $\delta' > 0$ with the following property. For every embedded  {\it real } $C^{1+Lip}$  $n$-dimensional manifold  $M \subset V$ satisfying 
$$diam \ M + dist_{\mathcal{B}_0}(M,G) < \delta'$$
for some $G \in \Omega$ and such that $M$ is transversal to $E^h_G$ ($T_x M \oplus  E^h_G=\mathcal{B}_0$ for every $x \in M$)  we have that $M\cap W^s_\delta(\Omega)$ has empty interior in $M$. If $\mathcal{B}_0$ is a complex Banach space and $\mathcal{R}$ is complex analytic with $n= \dim_{\mathbb{C}} E^u_G$ then we say that  $W^s_\delta(\Omega)$ has {\bf complex  transversal empty interior} if there exists some $\delta' > 0$ such that for every embedded  {\it complex analytic }  $n$-dimensional manifold  $M \subset V$ the above property holds.

Using a notation similar to  J. Lindenstrauss and D. Preiss \cite{lind} \cite{lind2}, let $\mathcal{B}$ be a Banach space. Consider   $I=[-1,1]$ with the normalized Lebesgue measure $m(a,b)=(b-a)/2$. Endow $T=I^{\mathbb{N}}$ with the product topology and the product Lebesgue measure $m$. Let $\Gamma^k(\mathcal{B})$, $k \in \mathbb{N}\cup \{\infty\}$,  be the set of all continuous functions 
$$\gamma\colon T \mapsto \mathcal{B}$$  with continuous partial derivatives
$$\lambda\in T \mapsto \partial^j_{i_1 i_2 \cdots i_j} \gamma(\lambda)$$
for every $j\leq k$, with  $i_1, \dots,i_j \in \mathbb{N}$ and $\lambda=(\lambda_i)_{i\in \mathbb{N}} \in T $ satisfying $\lambda_{i_p}  \in (-1,1)$ for every  $p\leq j$. Moreover the partial derivatives extends continuously to a function in $T$. The pseudo-norms given  by the supremum of $\gamma$ and its partial derivatives up to order $k$ on $T$ endow $\Gamma^k(\mathcal{B})$ with the structure of a  Fr\'echet space. We have that $\Gamma^k(\mathcal{B})$ is  a Polish space, and consequently a Baire Space. Note that $\Gamma^1(\mathcal{B})$ coincides with $\Gamma(\mathcal{B})$ as defined in \cite{lind}. 

If $\mathcal{B}$ is a complex  Banach space, we will denote by  $C^{\omega}(\overline{\mathbb{D}}^j,\mathcal{B})$ the set of complex  analytic functions 
$$\gamma\colon \mathbb{D}^j \rightarrow \mathcal{B},$$ 
that have  a continuous extension to $\overline{\mathbb{D}}^j$. Endowed with the sup norm on $\overline{\mathbb{D}}^j$ we have that $C^{\omega}(\overline{\mathbb{D}}^j,\mathcal{B})$ is a complex Banach space. 

If $\mathcal{B}$ is  a complex Banach space, we want to consider $\Gamma^{\omega}(\mathcal{B})$. To do this, replace in the definition of $T$ the interval $[-1,1]$ by $\overline{\mathbb{D}}=\{ z \in \mathbb{C}\colon \  |z|\leq 1\}$ endowed with the normalized Lebesgue measure, obtaining $T_{\mathbb{C}}$. Then $\Gamma^{\omega}(\mathcal{B})$ is the space of all continuous functions $\gamma\colon T_{\mathbb{C}} \mapsto \mathcal{B}$ that are holomorphic when we fix all except a finite number of entries of $\lambda \in T_{\mathbb{C}}$ and $|\lambda_i| < 1$ for the remaining ones. The sup norm on $\gamma$ turns $\Gamma^{\omega}(\mathcal{B})$ into a complex Banach space. 

Finally we would like to consider real analytic families  into a real Banach space  $\mathcal{B}$. Denote by $\mathcal{B}_\mathbb{C}$ the  complex Banach space that is a complexification of $\mathcal{B}$ endowed with   a {\it desirable} norm as defined in Kirwan \cite{desirable} (see also Mu{\~n}oz,  Sarantopoulos and  Tonge \cite{normcomp}). We will denote by  $C^{\omega_{\mathbb{R}}}([-1,1]^j,\mathcal{B})$ the set of real analytic functions (see Bochnak and Siciak \cite{complex}).
$$\gamma\colon (-1,1)^j \rightarrow \mathcal{B}$$ 
that can be extended to a complex analytic function  
$$\gamma\colon \mathbb{D}^j \rightarrow \mathcal{B}_\mathbb{C},$$ 
and moreover $\gamma$ has a continuous extension to $\overline{\mathbb{D}}^j$. Endowed with the sup norm on $\overline{\mathbb{D}}^j$ we have that $C^{\omega_{\mathbb{R}}}([-1,1]^j,\mathcal{B})$ is a real Banach space.  

We define $\Gamma^{\omega_{\mathbb{R}}}(\mathcal{B})$ as the {\it real } Banach space  that consists of the restrictions to $T=[-1,1]^{\mathbb{N}}$ of the  functions $\gamma \in \Gamma^{\omega}(\mathcal{B}_{\mathbb{C}})$ that satisfy $\gamma(\overline{\lambda})= \overline{\gamma(\lambda)}$.

Our first main result is

\begin{thm}\label{main} Let $\mathcal{R}$ be a strongly compact $C^2$ map  on a real Banach space $\mathcal{B}_0$ with a compact invariant partially hyperbolic set $\Omega$  such that $\mathcal{R}\colon \Omega\rightarrow \Omega$ is onto. If  \begin{itemize}
\item[A.] There is $\delta > 0$   such that $W^s_\delta(\Omega)$ has  real transversal empty interior,
\item[B.]  For every $i\geq 0$ and every $ x\in \mathcal{R}^{-i}W^s_\delta(\Omega)$ we have that $D_x \mathcal{R}^i$ has dense image,
\end{itemize}
then there exists $\delta_5 >0$ such that $\cup_{i\in \mathbb{N}} \mathcal{R}^{-i} W^s_{\delta_5}(\Omega)$ is a  $\Gamma^k(\mathcal{B}_0)$-null set, for every $k\in \mathbb{N}\cup\{\infty, \omega_{\mathbb{R}}\}$. Indeed  for every $j$ and $k\in \mathbb{N}\cup\{\infty, \omega_{\mathbb{R}}\}$   there exists a residual set of  $C^k$  maps 
$$\gamma\colon [-1,1]^j\rightarrow \mathcal{B}_0$$ such that 
$$m(t \in [-1,1]^j \colon \ \gamma(t) \in \cup_{i\in \mathbb{N}} \mathcal{R}^{-i} W^s_{\delta_5}(\Omega))=0.$$
Here $m$ denotes the normalized Lebesgue measure on $[-1,1]^j$.  \end{thm} 

If $\mathcal{R}$ is complex analytic and $W^s_\delta(\Omega)$ has {\it real } transversal empty interior we can apply Theorem \ref{main}.  We do not know if { \it complex } transversal empty interior implies the { \it real } transversal empty interior property. This is the motivation to 

\begin{thm}\label{main2} Let $\mathcal{R}$ be a strongly compact, complex analytic map  on a complex  Banach space $\mathcal{B}_0$ with a compact invariant partially hyperbolic set $\Omega$  such that $\mathcal{R}\colon \Omega\rightarrow \Omega$ is onto. If  \begin{itemize}
\item[A.] There is $\delta > 0$   such that $W^s_\delta(\Omega)$ has complex  transversal empty interior,
\item[B.]  For every $i\geq 0$ and every $ x\in \mathcal{R}^{-i}W^s_\delta(\Omega)$ we have that $D_x \mathcal{R}^i$ has dense image,
\end{itemize}
then there exists $\delta_5 >0$ such that $\cup_{i\in \mathbb{N}} \mathcal{R}^{-i} W^s_{\delta_5}(\Omega)$ is a  $\Gamma^\omega(\mathcal{B}_0)$-null set. Indeed  for every $j$    there exists a residual set of  maps $\gamma \in C^{\omega}(\overline{\mathbb{D}}^j,\mathcal{B}_0)$   such that 
$$m(t \in \overline{\mathbb{D}}^j \colon \ \gamma(t) \in \cup_{i\in \mathbb{N}} \mathcal{R}^{-i} W^s_{\delta_5}(\Omega))=0.$$
Here $m$ denotes the Lebesgue measure on $\overline{\mathbb{D}}^j$. Moreover, if we consider  $\mathcal{B}_0$  as a real Banach space then all conclusions of Theorem \ref{main} holds.\end{thm}

\begin{rem} We will replace the transversal empty condition for a weaker assumption that is a little more technical. We postpone the description of this condition to the end of Section \ref{adapted}.  \end{rem} 

\begin{rem} If $\Omega$ is a hyperbolic invariant set for $\mathcal{R}$ and $\delta_5$ is small enough then $\cup_{i\in \mathbb{N}} \mathcal{R}^{-i} W^s_{\delta_5}(\Omega) $ is  the global stable lamination $W^s(\Omega)$.  \end{rem}

One can  ask if  we could extend our results to more general contexts that already appeared in the finite-dimensional setting. For instance, we could certainly generalize  our result assuming that $\mathcal{R}$ satisfies the  nonuniformly expanding condition on the centre-unstable direction as in Alves and Pinheiro \cite{ap}. However we feel that the most crucial adaptations to the infinite dimensional setting appeared  already in the present case and that the additional modifications of  their methods to achieve these results would be minor.

\section{Null sets in Banach spaces.} \label{nullsection}

Define $\Gamma^k_j(\mathcal{B}) \subset \Gamma^k(\mathcal{B})$ as the subset of all functions $\gamma$  where $\gamma(\lambda)$  depends only on the first $j$ entries of $\lambda$.  Then 
$$\cup_{j \in \mathbb{N}} \Gamma^k_j$$
is dense in $\Gamma^k(\mathcal{B})$. The proof is left to the reader (see the case $k=1$ in \cite{lind}).  If $\Theta \subset \mathcal{B}$ is a borelian set and $\gamma \in \Gamma^k(\mathcal{B})$, denote
$$m_\gamma(\Theta)=m(\lambda \in T\colon \gamma(\lambda) \in \Theta),$$
where $m$ is the product Lebesgue measure on $T$. Note that $m$ and $m_\gamma$  are regular Borel measures. We say that $\Theta$ is a {\it $\Gamma^k(\mathcal{B})$-null set}  is there exists a residual subset $\mathcal{F} \subset \Gamma^k(\mathcal{B})$ such that $m_\gamma(\Theta)=0$ for every $\gamma \in \mathcal{F}$. J. Lindenstrauss and D. Preiss \cite{lind} observed  that if $\mathcal{B}$ has finite dimension then a borelian $\Theta$ is a  $\Gamma^1(\mathcal{B})$-null set if and only if $\Theta$ has zero Lebesgue measure.

All results of this section hold for every $k \in \mathbb{N} \cup \{ \infty, \omega,\omega_{\mathbb{R}}\}$, but keep in mind that whenever we consider  $k=\omega$ the space  $\mathcal{B}$ is a complex Banach space and in the case $k=\omega_{\mathbb{R}}$ the real  Banach space $\mathcal{B}$ is the real trace of a complex Banach space $\mathcal{B}_\mathbb{C}$ and the definition of $\Gamma^{\omega_{\mathbb{R}}}(\mathcal{B})$ depends on  $\mathcal{B}_\mathbb{C}$. We omit the proofs of  Lemma \ref{densenull}, Lemma \ref{densenull2} and Proposition \ref{densenull3}  below for $k=\omega, \omega_\mathbb{R}$. The proofs in these cases are quite similar and indeed easier, since we do not need to deal with partial derivatives.

\begin{lem}\label{densenull} Let $\Theta$ be a countable union of closed subsets  of the Banach space $\mathcal{B}$. If there is a dense  subset $\mathcal{S} \subset \Gamma^k(\mathcal{B})$ such that $m_\gamma(\Theta)=0$ for $\gamma \in \mathcal{S}$ then $\Theta$ is a $\Gamma^k(\mathcal{B})$-null set.
\end{lem}
\begin{proof} Since a countable union of $\Gamma^k(\mathcal{B})$-null sets is a $\Gamma^k(\mathcal{B})$-null set, it is enough to prove the lemma assuming that $\Theta$ is a closed set. Let $\gamma \in \mathcal{S}$. Then 
$$\Theta_\gamma = \{ \lambda \in T \colon \gamma(\lambda) \in \Theta\}$$
is a compact set with zero product Lebesgue measure. Let $O \subset T$ be an open subset such that $\Theta_\gamma \subset O$ and $m(O) < \epsilon$.  Let $\delta = dist_{\mathcal{B}}(\gamma(O^c), \Theta) > 0.$  If $\tilde{\gamma} \in \Gamma^k(\mathcal{B})$ satisfies 
$$\sup_{\lambda \in T} |\gamma(\lambda)-\tilde{\gamma}(\lambda)| < \delta/2$$
then $ \Theta_{\tilde{\gamma}} \subset O$, so $m_{\tilde{\gamma}}(\Theta)< \epsilon$. In particular
$$\mathcal{S}_\epsilon = \{ \tilde{\gamma} \in \Gamma^k(\mathcal{B})\colon \ m_{\tilde{\gamma}}(\Theta)< \epsilon \}$$
contains an open and dense  subset of $\Gamma^k(\mathcal{B})$. Since
$$\mathcal{F}= \{  \tilde{\gamma} \in \Gamma^k(\mathcal{B})\colon \ m_{\tilde{\gamma}}(\Theta)=0 \}= \cap_{n \in \mathbb{N}^\star} \mathcal{S}_{1/n}$$
we conclude that $\mathcal{F}$ is a residual subset of  $\Gamma^k(\mathcal{B})$, so $\Theta$ is a $\Gamma^k(\mathcal{B})$-null set.
\end{proof}

\begin{lem}\label{densenull2}  Let $\Theta$ be  $\Gamma^k(\mathcal{B})$-null  borelian set. Then  for every $j$  there is dense  subset $\mathcal{S}_j \subset \Gamma^k_j(\mathcal{B})$ such that $m_\gamma(\Theta)=0$ for  every $\gamma \in \mathcal{S}_j$.
\end{lem}
\begin{proof} Let $A \subset   \Gamma^k(\mathcal{B})$ be an open set such that $A\cap \Gamma^k_j(\mathcal{B})\neq \emptyset $. Let $\gamma \in A\cap \Gamma^k_j(\mathcal{B})$. Then there exists $n$  vectors $a_\ell=(i_1^\ell, i_2^\ell, \cdots , i^\ell_{j_\ell})$, with $j_\ell \leq k$, $\ell\leq n$, such that  if 
 $$\partial_{a_\ell}=\partial^{j_\ell}_{i_1^\ell i_2^\ell \cdots i_{j_\ell}^\ell} $$ then
 $$A'=\{\beta \in \Gamma^k(\mathcal{B})\colon \ |\partial_{a_\ell} (\beta-\gamma)|_\infty < \epsilon, \  for \ every \  \ell\leq  n   \} \subset A.$$
Since $\Theta$ is a $\Gamma^k(\mathcal{B})$-null set there exists $\alpha \in A'$ such that $m_\alpha(\Theta)=0$. For every $\omega=(t_{j+1},t_{j+2},\dots)\in I^{\mathbb{N}}$ define the  finite dimensional subfamily $\alpha_{\omega} \in \Gamma^k_j(\mathcal{B})$  as
$$\alpha_{\omega}(t_1,\dots,t_j)= \alpha(t_1,\dots, t_j,t_{j+1},\dots).$$
Note that  $\alpha_{\omega} \in A'$. By the Fubini's Theorem for almost every $ \omega \in I^{\mathbb{N}}$ we have that $m_{\alpha_\omega}(\Theta)=0$.  This proves the lemma. 
\end{proof}

It was noted  by J.~Lindenstrauss, D.~Preiss, and J.~Ti{\v{s}}er \cite{lind2} in the case $k=1$ that if $\Theta$ is a countable union of closed subsets we have that various concepts of null-sets coincides. Indeed for arbitrary $k \in \mathbb{N} \cup \{ \omega,\omega_{\mathbb{R}}\}$  we have 
\begin{prop}\label{densenull3} Let $\Theta$ be a countable union of closed subsets  in the Banach space $\mathcal{B}$. The following statements are equivalent. 
\begin{itemize}
\item[A.] The set  $\Theta$ is a $\Gamma^k(\mathcal{B})$-null  set.
\item[B.]  For every $j$  there is a dense  subset $\mathcal{S}_j \subset \Gamma^k_j(\mathcal{B})$ such that $m_\gamma(\Theta)=0$ for  every $\gamma \in \mathcal{S}_j$.
\item[C.] For every $j$  there is a  residual   subset $\mathcal{S}_j \subset \Gamma^k_j(\mathcal{B})$ such that $m_\gamma(\Theta)=0$ for  every $\gamma \in \mathcal{S}_j$.
\end{itemize}
\end{prop} 
\begin{proof} The implication $(A)\Rightarrow (B)$ follows from Lemma \ref{densenull2}. Using the same argument of the proof of Lemma \ref{densenull} (replacing $\Gamma^k(\mathcal{B})$ by $\Gamma^k_{j}(\mathcal{B})$) one can easily show that $(B)\Rightarrow (C)$.  To show that $(C)\Rightarrow (A)$, note that since $\cup_j \Gamma^k_j(\mathcal{B})$ is dense in $\Gamma^k(\mathcal{B})$, item $C$ implies that there is a dense  subset $\mathcal{S} \subset \Gamma^k(\mathcal{B})$ such that $m_\gamma(\Theta)=0$ for $\gamma \in \mathcal{S}$. By Lemma \ref{densenull} we get $A$. 

\end{proof}

\begin{prop}\label{generic} Suppose that 
\begin{itemize} 
\item[(H)] The set $\Theta$ is  a borelian  subset  of the Banach space $\mathcal{B}$ such that for every $x \in \mathcal{B}$ there exists $\delta=\delta(x) > 0$ and a finite dimensional  family $\alpha \in \cup_j \Gamma^k_j(\mathcal{B})$ with $\alpha(0)=x$ such that for every $z\in \mathcal{B}$  satisfying $|z|< \delta$ we have $m_{\alpha_z}(\Theta)=0$, where 
$$\alpha_z(\lambda) = \alpha(\lambda)+ z,$$ 
\end{itemize} 
Then for every $j\geq 0$ and every $k \in \mathbb{N}\cup \{\infty, \omega, \omega_\mathbb{R}\}$ there is a dense subset  $\mathcal{S}_{j}$ of  $\Gamma^k_{j}(\mathcal{B})$ such that $m_\gamma(\Theta)=0$ for every $\gamma \in \mathcal{S}_{j}$.  
\end{prop}
\begin{proof} Suppose $k\neq \omega$.  Let $A \subset \Gamma^k(\mathcal{B})$  be open subset such that  there exists $\gamma \in A\cap \Gamma^k_{j}(\mathcal{B})$. By the compactness of $[-1,1]^{j}$ and continuity of $\gamma$  there exists a finite subset $\{\lambda_1, \dots, \lambda_p \} \subset [-1,1]^{j}$ and $\epsilon_i >0$, with $i\leq p$, satisfying 
$$[-1,1]^{j}\subset \cup_i B(\lambda_i,\epsilon_i),$$
and such that $|\lambda-\lambda_i| < \epsilon_i$ implies $|\gamma(\lambda)-\gamma(\lambda_i)| < \delta(\gamma(\lambda_i))/2$. Moreover there are  families $\alpha_i \in  \Gamma^k_{j_i}(\mathcal{B})$ such that $\alpha_i(0)=\gamma(\lambda_i)$ and $m_{\alpha_{i,z}}(\Theta)=0$ for every   $|z| <  \delta(\gamma(\lambda_i))$. Here 
$$\alpha_{i,z}(t) = \alpha_i(t)+ z,$$ 
Let $J= j+\sum_i j_i$. Define $\beta \in \Gamma^k_{J}(\mathcal{B})$ as
$$\beta(t,t_1,\dots, t_p) = \gamma(t) + \sum_{i=1}^p\alpha_i(t_i)-\gamma(\lambda_i).$$
Here $t_i \in [-1,1]^{j_i}$ and $t \in [-1,1]^{j}$.  Choose $\eta> 0$ small enough such that if $|t_i| < \eta$ for every $i\leq p$ then 
$$\sum_{i=1}^p |\alpha_i(t_i)-\gamma(\lambda_i)| < \min_i \delta(\gamma(\lambda_i))/2.$$
Given $i_0 \leq p$, fix the entries $t, t_1,\dots,t_{i_0-1},t_{i_0+1},\dots, t_p$, with $|t -\lambda_{i_0}|< \epsilon_{i_0}$ and $|t_i|< \eta$ and   consider the family $\beta$ restricted to the segment 
$$L=\{(t, t_1,\dots, t_{i_0-1},t_{i_0}, t_{i_0+1},\dots, t_p  )  \in T\colon \ |t_{i_0}|< \eta \}$$
Note that in this line
$$\beta(t,t_1,\dots, t_p)= \alpha_{i_0}(t_{i_0})+ z,$$
where 
$$z = \gamma(t)-\gamma(\lambda_{i_0}) + \sum_{i\neq i_0} \alpha_i(t_i)-\gamma(\lambda_i),$$
Since $|z|<   \delta(\gamma(\lambda_{i_0})) $  it follows that $m_{\alpha_{i_0}+z}(\Theta)=0$. By  Fubini's Theorem we have that the set of parameters
$$u \in \{ t \in [-1,1]^j\colon |t- \lambda_{i_0}| < \epsilon_{i_0}\} \times_i \{ t_i \in [-1,1]^{j_i}\colon \ |t_i|< \eta\} $$
such that $\beta(u) \in \Theta$  has zero Lebesgue measure. This holds for every $i_0\leq p$, so we  conclude that the set of parameters
$$u \in  [-1,1]^j  \times_i \{ t_i \in [-1,1]^{j_i}\colon \ |t_i|< \eta\} $$
such that $\beta(u) \in \Theta$  has zero Lebesgue measure.
 Now we can apply Fubini's Theorem once again to conclude that for almost every 
$$(t_1,\dots, t_p) \in \times_i \{ t_i \in [-1,1]^{j_i}\colon \ |t_i|< \eta\} $$
the family $\beta_{t_1,\dots, t_p }\colon [-1,1]^j\mapsto \mathcal{B}$ given by 
$$\beta_{t_1,\dots, t_p }(t)= \beta(t,t_1,\dots, t_p)$$
satisfies $m_{\beta_{t_1,\dots, t_p}}(\Theta)=0$. But $\beta_{t_1,\dots, t_p}$ is a translation of the family  $\gamma$, that is 
$$\beta_{t_1,\dots, t_p }(t)= \gamma(t) + z_{t_1,\dots, t_p },$$
where
$$\lim_{(t_1,\dots,t_p)\rightarrow 0} z_{t_1,\dots, t_p } =0,$$
so we can choose  $(t_1,\dots,t_p)$ sufficiently small such that $$\beta_{t_1,\dots, t_p } \in A \ and \ m_{\beta_{t_1,\dots, t_p}}(\Theta)=0.$$
The proof for the case $k=\omega$ is obtained replacing $[-1,1]$ by $\overline{\mathbb{D}}$ everywhere. 
\end{proof}

\begin{cor} Under  assumption (H)  in  Proposition \ref{generic}, if $\Theta$ is also a countable union of closed  sets then $\Theta$ is a $\Gamma^k(\mathcal{B})$-null set.
\end{cor}
\begin{proof} Since $\cup_j \Gamma^k_j(\mathcal{B})$ is dense in $\Gamma^k(\mathcal{B})$, by Proposition \ref{generic} the set $\cup_j \mathcal{S}_j$ is dense in $\Gamma^k(\mathcal{B})$. Now we can apply Lemma \ref{densenull} to conclude that $\Theta$ is a  $\Gamma^k(\mathcal{B})$-null set.\end{proof}

\begin{rem} \label{ka} We can obtain  similar results considering    either the weak or strong Whitney topology on $C^k((-1,1)^j,\mathcal{B})$, the  space of all $C^k$ functions $$\gamma\colon (-1,1)^j \rightarrow \mathcal{B}.$$ 
Indeed if  $\Theta$ is a countable union of closed sets satisfying assumption (H)  in  Proposition \ref{generic} then for a residual set of functions $\gamma \in  C^k((-1,1)^j,\mathcal{B})$ we have $m_{\gamma}(\Theta)=0$. 
\end{rem}

\section{Adapted norms, cones and transversal families.}\label{adapted}
Assume that $\mathcal{R}$ and $\Omega$ satisfy the assumptions of either Theorem \ref{main} or Theorem \ref{main2}, so we are going to deal with the real smooth and complex analytic cases at the same time. Replacing $|v|_1$ by the equivalent norm $|v|_1+ |v|_0$, we can assume without loss of generality that $|v|_1\geq |v|_0$. Denote  by $\pi_G^{h}$ and $\pi_G^u$ the continuous  linear projections  on $E^{h}_G$ and $E^u_G$ such that 
$$v = \pi_G^u(v) + \pi_G^{h}(v).$$
Since $E^u_G \subset \mathcal{B}_1$ we have that $\pi_G^h(\mathcal{B}_1)\subset \mathcal{B}_1$ and that 
$$\pi_G^u, \pi_G^h\colon \mathcal{B}_i\rightarrow \mathcal{B}_i$$
are bounded operators. Define the adapted norms
$$|v |_{G,i} = |\pi^u_G(v)|_i+ |\pi^h_G(v)|_i.$$
Here $i=0,1$.  Replacing $\mathcal{R}$ by some iteration of it we have the following properties
\begin{itemize}
\item[i.] We have that 
$$G \in \Omega \mapsto \pi_G^u $$
and
$$G \in \Omega  \mapsto \pi_G^{h}$$
are continuous with respect to the  $\mathcal{B}_0$ topology.
\item[ii.]  We have  $|v|_{G,0}\leq |v|_{G,1}$ for every $v \in \mathcal{B}_1$ and 
$$|\pi_G^u(v)|_{G,i}\leq |v|_{G,i} \ and \ |\pi_G^{h}(v)|_{G,i}\leq |v|_{G,i}.$$
\item[iii.] There exist $\delta_1> 0$ and $\lambda_G > 0$ such that if  $ v\in E^{h}_G$ then  
$$|D_F\mathcal{R}\cdot v|_{\mathcal{R}G,1}\leq \lambda_G  |v|_{G,0} $$
provided $|F-G|_{G,0} < \delta_1$. 
\item[iv.] There exists $\theta > 1$ such that for every $G \in \Omega$  there exists $\theta_G > \theta \max \{ 1,  \lambda_G \}$ such that if $v \in E^u_G$
$$|D_F\mathcal{R}\cdot v|_{\mathcal{R}G,0}\geq  \theta_G  |v|_{G,0}. $$
provided $|F-G|_{G,0} < \delta_1$. 
\item[v.]  There exists $\Cl{comp44}  > 1$ such that for every $G \in \Omega$ and $v \in E^u_G$ 
$$|v|_{G,1}\leq \Cr{comp44} |v|_{G,0}.$$
\item[vi.] There exists $\Cl{comp} > 1$ such that for every $G \in \Omega$ 
\begin{equation}\label{comp78} \frac{1}{\Cr{comp}} |v|_{i} \leq |v|_{G,i}\leq \Cr{comp} |v|_{i}.\end{equation}
\item[vii.] There exists $\Cl{c55} > 1$ such that for every $G, F \in \Omega$  and $ v\in \mathcal{B}_1$ we have 
\begin{equation}\label{comp70}  |\pi^h_G(v)|_{G,1}\leq \Cr{c55} |v|_{F,1}.\end{equation}
\end{itemize}
\begin{rem} Since $\mathcal{R}$ is a strongly compact $C^2$ map, it follows that 
$$D_G\mathcal{R}\colon \mathcal{B}_0 \rightarrow  \mathcal{B}_0 $$ is a compact linear operator. Indeed  $Im(D_G\mathcal{R})\subset \mathcal{B}_1$ and 
$$D_G\mathcal{R}\colon \mathcal{B}_0 \rightarrow  \mathcal{B}_1 $$
is a bounded linear transformation. 
If $G \in \Omega$ then $D_G\mathcal{R}$ is a compact linear isomorphism between $E^u_G$ and $E^u_{\mathcal{R}G}$. In particular $E^u_G$ has finite dimension. 
\end{rem}

Indeed (i) follows from the continuity of the distributions $E^h_G$ and $E^h_g$ and (ii)  follows from  the definition of  the adapted norms.   Properties (iii) and (iv) follows from the fact that $\Omega$ is a partially hyperbolic invariant set for $\mathcal{R}$ and that $\mathcal{R}$ is a strongly compact operator.  To prove (v), recall that $\mathcal{R}\colon \Omega \rightarrow \Omega$ is onto, so there exists $F \in \Omega$ such that $\mathcal{R}F=G$. By (iii), since $E^u_G$ is finite dimensional, $E^u_G \subset \mathcal{B}_1$ and $D_{F}\mathcal{R}(E^u_{F})=E^u_G$ for every $v \in   E^u_G$ let $w \in E^u_{F}$ be such that $D_{F}\mathcal{R}\cdot w = v$. By (iv) we have 
$$ |w|_{F,0}  \leq \frac{1}{\theta}  |v|_{G,0} .$$
Since $\mathcal{R}$ is $C^1$ and $\Omega$ is compact, there exists $\Cl{normder}$ such that $$\sup_{F\in \Omega}  |D_G \tilde{\mathcal{R}} |_{(\mathcal{R}F,1), (F,0)} < \Cr{normder},$$
so
$$|v|_{G,1}= |D_F \tilde{\mathcal{R}} \cdot w|_{G,1}\leq \Cr{normder} |w|_{F,0} \leq \frac{\Cr{normder}}{\theta}  |v|_{G,0}.$$

We are going to  prove (vi). The case $i=0$ follows from (i) and the compactness of $\Omega$. Then by (v) and the case $i=0$
$$|\pi^u_G(v)|_{G,1}=|\pi^u_G(v)|_{1}\leq   \Cr{comp44}  |\pi^u_G(v)|_{G,0}\leq \Cr{comp}  \Cr{comp44} |v|_{0}\leq \Cr{comp}  \Cr{comp44} |v|_{1}. $$
so
$$|\pi^h_G(v)|_{G,1}=|\pi^h_G(v)|_{1}= |v-\pi^u_G(v)|_{G,1}  \leq   |v|_1+ |\pi^u_G(v)|_{1}  \leq (1+ \Cr{comp}  \Cr{comp44}) |v|_{1}. $$
Consequently $|v|_{G,1}\leq (1+ 2  \Cr{comp}  \Cr{comp44}) |v|_1$. Of course  $|v|_1\leq  |v|_{G,1}.$ This proves (vi) for the case $i=1$. It remains to show (vii). Indeed by (v)  
$$ |\pi^h_G(v)|_{G,1} \leq |v|_{G,1} \leq \Cr{comp} |v|_{1}\leq \Cr{comp}^2 |v|_{F,1}.$$

Now we can define the  unstable cones  $C^u_{\epsilon,i}(G)$, $i=0,1$,  as the set of all $v \in \mathcal{B}_i$ such that $v = u+ w$, with $u \in E^u_G$ and $w\in E^{h}_G$ (note that this implies $u,w \in \mathcal{B}_i$) and moreover
$$|w|_{G,i}\leq \epsilon |u|_{G,i}.$$   Define the  cone $C^u_{\epsilon,(1,0)}(G)$ as the set of  $v \in \mathcal{B}_1$ such that $v = u+ w$, with $u \in E^u_G$ and $w\in E^{h}_G$ and 
$$|w|_{G,1}\leq \epsilon |u|_{G,0}.$$
Of course
$$C^u_{\epsilon,(1,0)}(G) \subset C^u_{\epsilon,0}(G) \cap C^u_{\epsilon,1}(G).$$
and  if $\epsilon' < \epsilon$ then 
$$C^u_{\epsilon',i}(G) \subset C^u_{\epsilon,i}(G).$$
Moreover these cones are forward invariant. Indeed
\begin{equation}\label{incones}D_G\mathcal{R} (C^u_{\epsilon,0}(G))   \subset  C^u_{\frac{\lambda_G}{\theta_G} \epsilon,(1,0)}(\mathcal{R}G) \subset C^u_{\frac{\epsilon}{\theta},0}(\mathcal{R}G)    \subset C^u_{\epsilon,0 }(\mathcal{R}G).\end{equation}
Since we are going to deal with many norms, it is convenient to introduce the following notation.  If $$T\colon L  \rightarrow  \mathcal{B}_i$$ is a linear transformation, where $L$ is a subspace of $\mathcal{B}_j$, with $i,j \in \{0,1\}$, and $G, G' \in \Omega$,   we denote
$$|T|_{(G,j),(G',i)}= \sup \{  |T(v)|_{G',i}\colon \  |v|_{G,j}\leq 1\}.$$
and if 
$$Q\colon L \times L  \rightarrow  \mathcal{B}_i$$
is a symmetric bilinear transformation then 
$$|Q|_{(G,j),(G',i)}= \sup \{  |Q(v,v)|_{G',i}\colon \  |v|_{G,j}\leq 1\}.$$
Define
$$|D\mathcal{R}|_\infty = \sup  \{ |D_F\tilde{\mathcal{R}}|_{(G,0),(\mathcal{R}G,1)}\colon \ G \in \Omega, |F-G|_{G,0}\leq \delta_1\},$$
$$|D^2\mathcal{R}|_\infty = \sup  \{ |D^2_F\tilde{\mathcal{R}}|_{(G,0),(\mathcal{R}G,1)}\colon \ G \in \Omega, |F-G|_{G,0}\leq \delta_1\}.$$
Fix    $\epsilon > 0$  small enough to satisfy
$$\theta_1=\min \{ \theta-2\epsilon, \frac{\theta-\epsilon}{1+\epsilon},  \theta\frac{ (1-\epsilon)^3}{(1+\epsilon)^2} , \big( \frac{(1+\epsilon)^2}{\theta^2 (1-\epsilon)^6}  + \frac{ \epsilon (1+\epsilon)  |D \mathcal{R}|_\infty (1+\epsilon)^3 }{\theta^3(1-\epsilon)^9}  \big)^{-1}     \} > 1.$$
In particular  $D_G\mathcal{R}$ is expanding on unstable cones.  Indeed, if $v \in C^u_0(G):=C^u_{\epsilon,0}(G)$ then 
$$|D_F\mathcal{R}\cdot v|_{\mathcal{R}G,0} \geq \frac{\theta_G (1-  \epsilon/\theta)}{1+\epsilon} |v|_{G,0} \geq  \Cl[e]{exp_cone} |v|_{G,0}$$ 
provided $|F-G|_{G,0} \leq \delta_1$. Let
$$d_\gamma= \sup \{ |D_F\mathcal{R}-D_G\mathcal{R}|_{(G,0),  (\mathcal{R}G,1)} \colon |F-G|_{G,0}\leq \gamma, G \in \Omega \}.$$
Choose $\delta_2 \in (0,\delta_1)$ small enough such that 
$$\Cl[c]{cone_contra}= \frac{\frac{1}{\theta} + \frac{d_{\eta_k}}{\epsilon \theta_1}}{1-\frac{\epsilon}{\theta} - \frac{d_{\delta_2}}{\theta_1}} < 1$$
and 
$$\frac{\Cr{cone_contra}}{1- \Cr{cone_contra} \epsilon} < 1.$$
This is possible because  $\theta - 2\epsilon > 1$. 
Then for every $F$ such that $|F- G|_{G,0}\leq \delta_2$, with $G \in \Omega$, we have
\begin{equation}\label{incones2}D_F\mathcal{R} (C^u_{\epsilon,0}(G))   \subset   C^u_{ \Cr{cone_contra} \epsilon,(1,0)}(\mathcal{R}G)    \subset C^u_{\epsilon,0 }(\mathcal{R}G).\end{equation}

Let   $G \in \Omega$.  Define 
$$\mathbb{B}(F,\delta) = \{\tilde{F} \in \mathcal{B}_0 \colon \ |\tilde{F}-F|_0 \leq \delta \},$$
$$\mathbb{B}^u_G(v_0,\delta) = \{v \in E^u_G \colon \ |v-v_0|_{G,0} \leq \delta \}$$
and 
$$E^h_G+G = \{v+G\colon \ v \in E^h_G\}.$$ 
Choose  $\delta_3  > 0$ small enough  such that $(2+\epsilon) \delta_3 < \delta_2$  and
\begin{equation}\label{compc5}  \Cl{uuu}= \sup  \{  |\tilde{\mathcal{R}}(F)|_1 \ s.t. \ dist_{\mathcal{B}_0}(F,\Omega) \leq  \Cr{comp} (2+\epsilon) \delta_3 \} < \infty.\end{equation} 

Let  $\mathcal{T}^k_i(G,\delta)$, $i=0,1$,  with $\delta \in (0, \delta_3)$, be  the set of all $C^k$ functions 
$$\mathcal{H}\colon \mathbb{B}^u_G(v_0,\delta) \rightarrow E^h_G\cap \mathcal{B}_i+G$$
such that 
$$|D\mathcal{H}|_{(G,0),(G,i)}\leq \epsilon$$
and
$$F_0 = v_0 +\mathcal{H}(v_0) \in W^s_{\delta_3}(G).$$
In particular
\begin{equation}\label{compc55} \sup_{v \in \mathbb{B}^u_G(v_0,\delta)} | v+ \mathcal{H}(v)-G|_{G,0}\leq (2+\epsilon)\delta_3 < \delta_2.\end{equation} 
We will call $F_0$ the {\it base  point} of $\mathcal{H}$. We will denote the graph of $\mathcal{H}$ by $\hat{\mathcal{H}}$ , that is, 
$$\hat{\mathcal{H}}= \{ v+ \mathcal{H}(v)\colon \  v \in \mathbb{B}^u_G(v_0,\delta)  \} .$$

The graph $\hat{\mathcal{H}}$ of a  function $\mathcal{H} \in  \mathcal{T}^k_i(G,\delta)$ will be called a {\bf transversal family}. See Figure 1. Note that 
 $$\mathcal{T}^k_1(G,\delta) \subset \mathcal{T}^k_0(G,\delta).$$
 
\begin{figure}
%if you want to compile with psfrag replace pdf by eps file and uncomment \psfrag lines. 
%\psfrag{C}{$C^u_0(G)$}
%\psfrag{g}[][][0.7]{$G \in \Omega$}
%\psfrag{F}[][][0.5]{$F_0=v_0+\mathcal{H}(v_0) \in W^s_{\delta_3}(G)$}
%\psfrag{es}{$E^h_G$}
%\psfrag{eu}{$E^u_G$}
%\psfrag{d}[][][0.7][58]{$\leq \delta_3$}
%\psfrag{B}[][][0.6][34]{$\mathbb{B}^u_G(v_0,\delta)$}
%\psfrag{v}[][][0.7]{$v_0$}
%\psfrag{u}[][][0.7]{$v$}
%\psfrag{iu}[][][0.7]{$v+\mathcal{H}(v)$}
%\includegraphics[scale=0.8]{transversal.eps}
\includegraphics[scale=0.4]{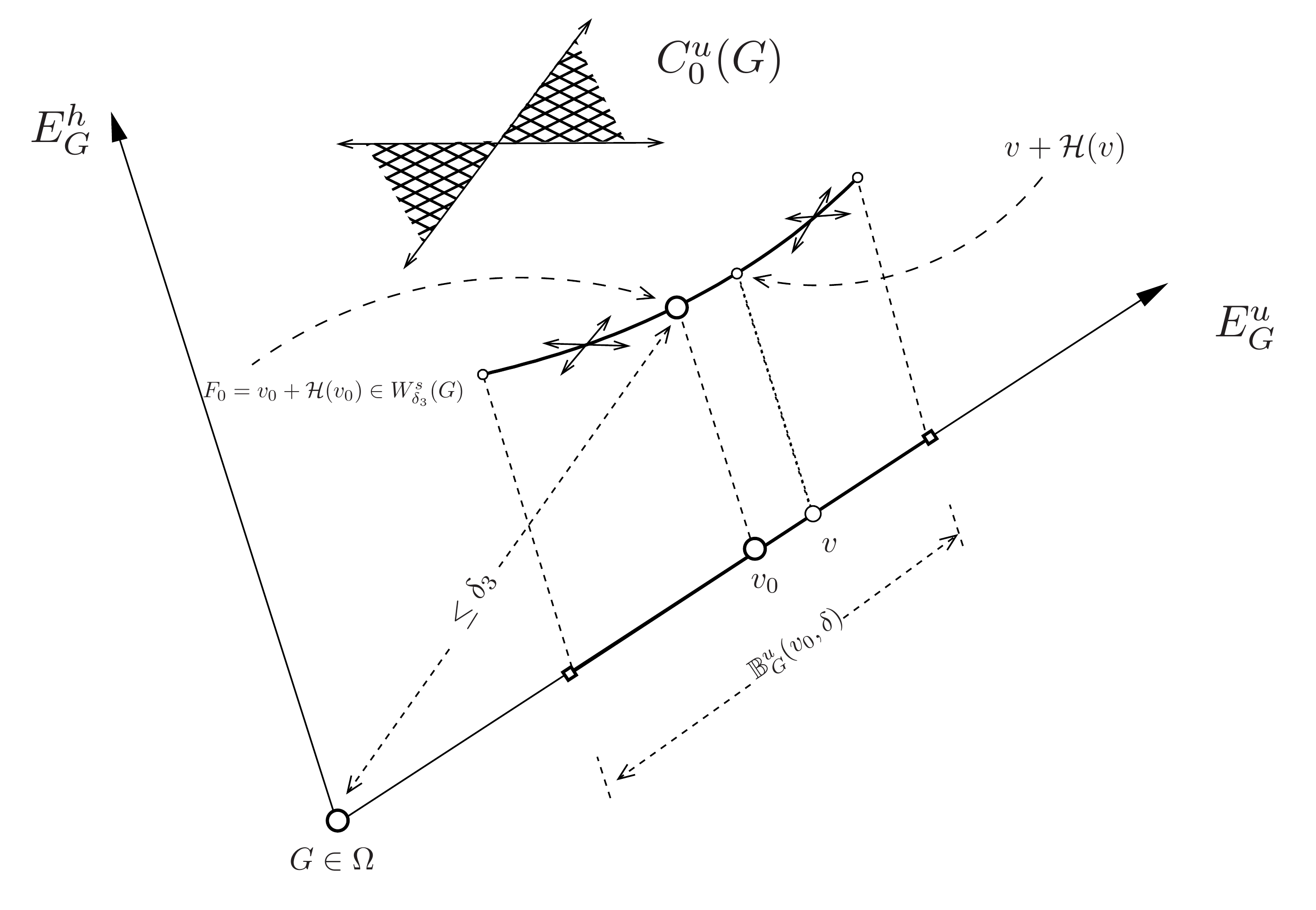}
\caption{How the graph of $\mathcal{H} \in \mathcal{T}^k_i(G,\delta)$  looks like. }
\end{figure}

In Theorems \ref{main}  and \ref{main2}  we assume that  $W^s_{\delta}(\Omega)$ has (either real or complex) transversal empty interior. But indeed  we just need that the Transversal Empty Interior Assumption holds for $n$-dimensional manifolds of the form  $\hat{\mathcal{H}}$. More precisely \\

\noindent {\bf Transversal Empty Interior Assumption}: There is $\delta > 0$ such that for every $\delta'$ small enough the following  holds.  For every $G \in \Omega$ and for every  $C^{1+Lip}$  function $\mathcal{H}\in \mathcal{T}^1_0(G,\delta')$ we have that $\hat{\mathcal{H}}\cap W^s_\delta(\Omega)$ has empty interior in  $\hat{\mathcal{H}}$.\\

If $\mathcal{R}$ and $\Omega$ satisfy the assumptions of Theorem \ref{main2}, then we interpret a $C^{1+Lip}$ function as a $C^{1+Lip}$ complex differentiable function, that is, a complex analytic function.

The definition of $\mathcal{T}^k_i(G,\delta)$ depends on $\delta_3 > 0$. However if  $\mathcal{U}^k_i(G,\delta)$  is a similar collection of functions obtained  replacing $\delta_3$ by a smaller positive value $\gamma < \delta_3$  then $\mathcal{U}^k_i(G,\delta) \subset  \mathcal{T}^k_i(G,\delta)$ and  the transversal empty interior assumption remains true. So we will replace $\delta_3$ by smaller positive values  a few times along this work and we will keep the same notation $\mathcal{T}^k_i(G,\delta)$ for the new family of functions $\mathcal{U}^k_i(G,\delta)$.

\section{Action of $\mathcal{R}$ on transversal families.} \label{action_t}

 In this section we will see  that not only the  operator $\mathcal{R}$ keeps invariant the set of  transversal families, but indeed,  if we keep track only of the piece of the transversal family close to the $\Omega$-limit set then the iterations of a transversal family consists in  bounded subset  in $C^2$ topology. This is similar to the well-know  $\lambda$-Lemma for hyperbolic points  due to  Palis \cite{palis1} (see also de Melo and Palis \cite{mp}). 
 
\begin{prop} \label{trans}  There are  $\Cr{22} > 1$, $\Cl{uu} > 0$  and  $ \Cl[c]{3}, \Cl[c]{1},  \Cl[c]{expa}  \in (0,1)$  with the following property.  For every  $\delta \in (0,\delta_3]$ and  every  $\mathcal{H}_0 \in \mathcal{T}^2_0(G,\delta)$, with $G \in \Omega$ and base  point $F \in W^s_{\delta_3}(G)$ the following holds: 
\begin{itemize} 

\item[A.] For every $x \in \mathbb{B}^u_{\mathcal{R}G}(v_1,\Cr{22}\delta)$, with $v_1 = \pi^u_{\mathcal{R}G}(\mathcal{R}F-\mathcal{R}G)$, there exists a unique $y \in E^h_{\mathcal{R}G}$ such that $\mathcal{R} G+x+y \in \mathcal{R}(\hat{\mathcal{H}}_0)$. 

\item[B.] Define  the function 
$$\mathcal{H}_1\colon \mathbb{B}^u_{\mathcal{R}G}(v_1,\Cr{22}\delta) \rightarrow E^h_{\mathcal{R}G}+\mathcal{R}G$$
as $\mathcal{H}_1(x)=\mathcal{R}G+y$, where $x, y$ are as in item A.  Then $\mathcal{H}_1 \in \mathcal{T}^2_1(\mathcal{R}G,\theta_1 \delta)$, where $\mathcal{R}F$ is a base point of $\mathcal{H}_1$. Moreover 
$$|D\mathcal{H}_1|_{(\mathcal{R}G,0),(\mathcal{R}G,1) }\leq \Cr{3}  \epsilon.$$ 

\item[C.] We have 
\begin{equation} \label{est55} |D^2 \mathcal{H}_1|_{(\mathcal{R}G,0),(\mathcal{R}G,1) } \leq  \Cr{uu} + {\Cr{1}}  |D^2 \mathcal{H}_0|_{G,0}.\end{equation}
\item[D.] We have that  $w+\mathcal{H}_1(w) \in \mathcal{B}_1$ for every $w \in \mathbb{B}^u_{\mathcal{R} G}(v_1,\theta_1\delta)$ and 
\begin{equation}|w+\mathcal{H}_1(w)|_{1} \leq \Cr{uuu}.   \end{equation}
\item[E.] There is a function 
$$\mathcal{R}^{-1}\colon  \hat{ \mathcal{H}}_1   \mapsto \hat{\mathcal{H}}_0$$
such that  $\mathcal{R}\circ \mathcal{R}^{-1}(Y)=Y$ for every $Y \in \hat{\mathcal{H}}_1$.  Its image is an open subset of $\hat{\mathcal{H}}_0$ and $\mathcal{R}^{-1}$   is  a diffeomorphism between $\hat{ \mathcal{H}}_1$ and $\mathcal{R}^{-1}(\hat{\mathcal{H}}_1)$. Moreover 
\begin{equation}\label{cinv} |\mathcal{R}^{-1}(F_1) - \mathcal{R}^{-1}(F_2)|_{G,0}\leq \Cr{expa} |F_1 - F_2|_{\mathcal{R}G,0}.\end{equation}
and  replacing $\mathcal{R}$ by a iteration of it  
\begin{equation}\label{cinv2} |\mathcal{R}^{-1}(F_1) - \mathcal{R}^{-1}(F_2)|_{0}\leq \Cr{expa} |F_1 - F_2|_{0}.\end{equation}
for every $F_1, F_2 \in \hat{\mathcal{H}}_1$.
\end{itemize}
\end{prop}

\begin{rem} If $\mathcal{R}$ and $\Omega$ satisfy the assumptions of Theorem \ref{main2}, one needs to consider $\mathcal{T}^2_0(G,\delta)$ as a set of $C^2$ complex differentiable functions, that is, complex analytic functions.  In this case  the estimates for the first and second derivatives of $\mathcal{H}_1$ are not necessary. \end{rem}

\begin{proof} Note that
\begin{equation}\label{d1}  (1-\epsilon) |v|_{G,0}\leq       |v + D_x \mathcal{H}_0\cdot v|_{G,0}\leq (1+\epsilon) |v|_{G,0}.\end{equation}
Let $\pi\colon \mathcal{B}_1 \mapsto \mathcal{B}_1$ be  a linear transformation and let 
$$\phi(x)= \pi(\tilde{\mathcal{R}}(x+\mathcal{H}_0(x)) - \tilde{\mathcal{R}}G).$$
Then
\begin{equation} \label{phi1} D_x\phi \cdot v = \pi[D_{x+\mathcal{H}_0(x)}\tilde{\mathcal{R}}\cdot(v + D_x \mathcal{H}_0\cdot v)],\end{equation}
\begin{equation} \label{phi2}  D_x^2 \phi \cdot v^2 =  \pi[ D^2_{x+\mathcal{H}_0(x)} \tilde{\mathcal{R}}\cdot(v + D_x \mathcal{H}_0\cdot v)^2 + D_{x+\mathcal{H}_0(x)}\tilde{\mathcal{R}}\cdot D^2_x\mathcal{H}_0\cdot v^2 ],\end{equation}
Suppose from now on that $|\pi|_{1}\leq 1$. We have
\begin{equation}  |D_x\phi|_{(G,0), (\mathcal{R}G,1)} \leq |D \tilde{\mathcal{R}}|_\infty (1+\epsilon).  \end{equation}
Moreover since 
$$D^2_x\mathcal{H}_0\cdot v^2  \in E^h_G,$$
and 
\begin{equation}\label{contra} |D_{x+\mathcal{H}_0(x)}\mathcal{R} \cdot w|_{\mathcal{R}G,1} \leq \lambda_G |w|_{G,0}\end{equation}
for every $w \in E^h_G$,  by (\ref{phi2}) and (\ref{contra}) we have 
\begin{eqnarray}\label{df1}
|D_x^2 \phi|_{(G,0), (\mathcal{R}G,1)}  &\leq&  |D^2\tilde{\mathcal{R}}|_\infty (1+\epsilon)^2 + \lambda_G |D^2 \mathcal{H}_0|_{G,0}  \nonumber \\
&\leq& \Cl{y}    +  \lambda_G |D^2 \mathcal{H}_0|_{G,0}.
\end{eqnarray}
Define
$$\phi_u(x)= \pi_{\mathcal{R}G}^u (\mathcal{R}(x+\mathcal{H}_0(x))-\mathcal{R}(G)).$$
$$\phi_h(x)= \pi_{\mathcal{R}G}^h (\mathcal{R}(x+\mathcal{H}_0(x))-\mathcal{R}(G)).$$
Note that for each $w  \in C^u_{(1,0)}(\mathcal{R}G)$ we have
\begin{equation}\label{d2} |\pi^u_{\mathcal{R}G}(w)|_{\mathcal{R}G,0} \geq (1-\epsilon) |w|_{\mathcal{R}G,0}\end{equation} 
and since $v + D_x \mathcal{H}_0\cdot v \in C^u_0(G)$ we have
\begin{equation}\label{d3} |D_{x+\mathcal{H}_0(x)}\mathcal{R}\cdot(v + D_x \mathcal{H}_0\cdot v)|_{\mathcal{R}G,0} \geq \frac{\theta_G (1-  \epsilon)}{1+\epsilon}  |v + D_x \mathcal{H}_0\cdot v|_{G,0}.\end{equation} 
It follows from (\ref{d1}), (\ref{d2}) and (\ref{d3}) that 
$$|D_x \phi_u\cdot v|_{\mathcal{R}G,0}\geq  \theta_G \frac{ (1-\epsilon)^3}{1+\epsilon} |v|_{G,0} \geq  \theta\frac{ (1-\epsilon)^3}{1+\epsilon} |v|_{G,0}  ,$$
so in particular $\phi_u$ is a diffeomorphism on $\mathbb{B}^u_G(v_0,\delta)$, its image contains $\mathbb{B}^u_{\mathcal{R}G} (v_1, \Cr{22} \delta)$, where $v_1 = \pi^u_{\mathcal{R}G}(\mathcal{R}F-\mathcal{R}G)$ and  $$\Cr{22}= \theta\frac{ (1-\epsilon)^3}{1+\epsilon}    > 1.$$ 
Consequently 
$$\phi^{-1}_u\colon \mathbb{B}^u_{\mathcal{R}G} (v_1,\Cr{22} \delta) \mapsto \mathbb{B}^u_G(v_0,\delta)$$
is  well defined and it is a  contraction on $ \mathbb{B}^u_{\mathcal{R}G} (v_1, \Cr{22} \delta)$, since 
\begin{equation}\label{df2} |D_x \phi^{-1}_u|_{(\mathcal{R}G,0),(G,0)}  \leq  \frac{1+\epsilon}{\theta_G (1-\epsilon)^3} \leq  \frac{1}{\Cr{22}}   < 1.\end{equation}
Note that since $\phi_u^{-1}\circ \phi_u(x)=x$ we have
$$D_{\phi_u(x)} \phi_u^{-1} \cdot D_x \phi_u \cdot w = w,$$
so if $w=(D_x\phi_u)^{-1}\cdot v$
\begin{equation} \label{h1} D_{\phi_u(x)} \phi_u^{-1} \cdot v  = (D_x\phi_u)^{-1}\cdot v.\end{equation}
Moreover 
\begin{equation} D_{\phi_u(x)}^2 \phi_u^{-1} \cdot (D_x \phi_u\cdot w)^2 + D_{\phi_u(x)} \phi_u^{-1}  \cdot D_x^2 \phi_u \cdot w^2 = 0,\end{equation} 
so
\begin{equation} \label{h2} D_{\phi_u(x)}^2 \phi_u^{-1} \cdot v^2 = - D_{\phi_u(x)} \phi_u^{-1}  \cdot D_x^2 \phi_u \cdot ((D_x\phi_u)^{-1}\cdot v)^2.\end{equation} 
By (\ref{df1})  and (\ref{df2}) 
\begin{eqnarray}
|D_{\phi_u(x)}^2 \phi_u^{-1}| _{(\mathcal{R}G,0),(G,0)} &\leq& \frac{(1+\epsilon)^3}{\theta_G^3 (1-\epsilon)^9}     \big( \Cr{y}    + \lambda_G |D^2 \mathcal{H}_0|_{G,0}. \big) \nonumber \\
&\leq& \Cl{xx} +   \frac{\lambda_G (1+\epsilon)^3  }{\theta_G^3 (1-\epsilon)^9}   |D^2 \mathcal{H}_0|_{G,0}. \nonumber \\
&\leq& \Cr{xx} +   \Cl[c]{segunda} |D^2 \mathcal{H}_0|_{G,0},
\end{eqnarray} 
where 
$$\Cr{segunda}= \frac{(1+\epsilon)^3  }{ \theta^3(1-\epsilon)^9} < 1$$
Note that for each $w  \in C^u_{(1,0)}(\mathcal{R}G)$ we have
\begin{equation}\label{d2h} |\pi^h_{\mathcal{R}G}(w)|_{\mathcal{R}G,1} \leq  \epsilon |w|_{\mathcal{R}G,0},\end{equation} 
so since
$$D_{x+ \mathcal{H}_0(x)} \mathcal{R}\cdot( v + D_x\mathcal{H}_0\cdot v) \in C^u_{(1,0)}(\mathcal{R}G)$$
we obtain 
$$|D\phi_h|_{(G,0),(\mathcal{R}G,1)}\leq \epsilon |D \mathcal{R}|_\infty  (1+\epsilon).$$
Then
$$\mathcal{H}_1 = \phi_h\circ \phi_u^{-1}+\mathcal{R}G$$
is well defined on $\mathbb{B}^u_{\mathcal{R}G} (v_1, \Cl[e]{22}  \delta)$. Note that

\begin{equation}\label{h11} D_x \mathcal{H}_1\cdot v  = D_{\phi_u^{-1}(x)} \phi_h \cdot D_x \phi_u^{-1} \cdot v.\end{equation}

\begin{equation}\label{h22} D_x^2 \mathcal{H}_1\cdot v^2  = D_{\phi_u^{-1}(x)}^2 \phi_h \cdot (D_x \phi_u^{-1} \cdot v)^2   + D_{\phi_u^{-1}(x)} \phi_h \cdot D_x^2 \phi_u^{-1} \cdot v^2.\end{equation}
Moreover
\begin{equation} \label{e33}  |D\mathcal{H}_1|_{(\mathcal{R}G,0),(\mathcal{R}G,1) } \leq \epsilon. 
\end{equation} 
Indeed, given $v$, let $w$ be
$$w= D_x \phi_u^{-1} \cdot v,$$
that is, if 
$$x=\phi_u(y)$$
then 
$$v= \pi_{\mathcal{R}G}^u(D_{y+\mathcal{H}_0(y)}\mathcal{R}\cdot(w + D_y \mathcal{H}_0\cdot w)).$$
Since
\begin{equation}\label{belong}  D_{y+\mathcal{H}_0(y)}\mathcal{R}\cdot(w + D_y \mathcal{H}_0\cdot w) \in C^u_{\Cr{cone_contra} \epsilon, (1,0) }(\mathcal{R}G)\end{equation}
we have that 
\begin{equation}\label{gb1} |v|_{\mathcal{R}G,0} \geq (1-\Cr{cone_contra} \epsilon)|D_{y+\mathcal{H}_0(y)}\mathcal{R}\cdot(w + D_y \mathcal{H}_0\cdot w)|_{\mathcal{R}G,0}.\end{equation}
Moreover
$$D_{\phi_u^{-1}(x)} \phi_h \cdot D_x \phi_u^{-1} \cdot v =  \pi_{\mathcal{R}G}^h(D_{y+\mathcal{H}_0(y)}\mathcal{R}\cdot(w + D_y \mathcal{H}_0\cdot w))$$
Due (\ref{belong}) we got 
\begin{equation}\label{gb2} |D_{\phi_u^{-1}(x)} \phi_h \cdot D_x \phi_u^{-1} \cdot v|_{\mathcal{R}G,1}\leq  \epsilon \Cr{cone_contra}  |D_{y+\mathcal{H}_0(y)}\mathcal{R}\cdot(w + D_y \mathcal{H}_0\cdot w)|_{\mathcal{R}G,0},\end{equation} 
so due (\ref{gb1})  and (\ref{gb2})
$$|D_{\phi_u^{-1}(x)} \phi_h \cdot D_x \phi_u^{-1} \cdot v|_{\mathcal{R}G,1} \leq \frac{ \epsilon \Cr{cone_contra}}{1-\epsilon \Cr{cone_contra}} |v|_{G,0}   $$
Recall we choose $\epsilon$ and $\delta_2$ such that 
$$ \Cr{3}:=  \frac{\Cr{cone_contra}}{1-\epsilon \Cr{cone_contra}}  < 1,$$
we obtain (\ref{e33}).  Indeed 
\begin{equation}  |D\mathcal{H}_1|_{(\mathcal{R}G,0),(\mathcal{R}G,1) } \leq \Cr{3} \epsilon. 
\end{equation} 
Note that $\mathcal{R}F \in W^s_{\delta_3}(\mathcal{R}G)$ and we have
$${\Cr{1}}=  \frac{(1+\epsilon)^2}{\theta^2 (1-\epsilon)^6}  + \frac{ \epsilon (1+\epsilon)  |D \mathcal{R}|_\infty (1+\epsilon)^3 }{\theta^3(1-\epsilon)^9}    < 1.$$
So by (\ref{h22}) we obtain
\begin{eqnarray}\label{h2e} |D_x^2 \mathcal{H}_1|_{(\mathcal{R}G,0),(\mathcal{R}G,1) } &\leq&   \frac{(1+\epsilon)^2}{\theta_G^2 (1-\epsilon)^6}  \big( \Cr{y}  
 +  \lambda_G |D^2 \mathcal{H}_0|_{G,0}  \big)  \nonumber \\   &+&   \epsilon (1+\epsilon)  |D \mathcal{R}|_\infty\big( \Cr{xx} +  \frac{(1+\epsilon)^3  }{ \theta^3 (1-\epsilon)^9}|D^2 \mathcal{H}_0|_{G,0}\big) \nonumber \\
 &\leq& \Cr{uu} +  \Big(  \frac{\lambda_G (1+\epsilon)^2}{\theta_G^2 (1-\epsilon)^6}  + \frac{ \epsilon  |D \mathcal{R}|_\infty (1+\epsilon)^4 }{\theta^3 (1-\epsilon)^9}   \Big)  |D^2 \mathcal{H}_0|_{G,0} \nonumber \\
 &\leq& \Cr{uu} + {\Cr{1}}  |D^2 \mathcal{H}_0|_{G,0}
 \end{eqnarray} 
 This proves A, B and C. Property D follows from (\ref{compc5}),(\ref{compc55}) and that $\hat{\mathcal{H}}_1 \subset \mathcal{R}\hat{\mathcal{H}}_0$. To show property $E$ define
  $$\mathcal{R}^{-1}(Y)=  \phi_u^{-1} \circ \pi^u_{\mathcal{R}G}(Y)   +  \mathcal{H}_0\circ  \phi_u^{-1} \circ \pi^u_{\mathcal{R}G}(Y)$$        
for every $Y \in \hat{\mathcal{H}}_1.$ Note that 
 $$\mathcal{R}(X)=    \phi_u\circ  \pi^u_G(X) + \mathcal{H}_1\circ  \phi_u\circ  \pi^u_G(X)$$
 for every $X \in \mathcal{R}^{-1}(\hat{\mathcal{H}}_1)$ so we have $\mathcal{R}\circ \mathcal{R}^{-1}(Y)=Y$ for every $Y \in \hat{\mathcal{H}}_1$.   Since $\phi_u$, $\phi_u^{-1}$, $\mathcal{H}_0$ and $\mathcal{H}_1$ are $C^1$ functions we conclude that $\mathcal{R}^{-1}$ is a diffeomorphism on $\hat{\mathcal{H}}_1$. Moreover for every $F_1,F_2 \in \hat{\mathcal{H}}_1$ we have by (\ref{df2}) 
 \begin{eqnarray} |\mathcal{R}^{-1}(F_1) - \mathcal{R}^{-1}(F_1) |_{G,0}&\leq&  (1+\epsilon)| \phi_u^{-1} \circ \pi^u_{\mathcal{R}G}(F_1) - \phi_u^{-1} \circ \pi^u_{\mathcal{R}G}(F_2)|_{G,0} \nonumber \\
 &\leq&  (1+\epsilon)| \phi_u^{-1} \circ \pi^u_{\mathcal{R}G}(F_1) - \phi_u^{-1} \circ \pi^u_{\mathcal{R}G}(F_2)|_{G,0} \nonumber \\ 
  &\leq&  \frac{1+\epsilon}{\Cr{22}}|\pi^u_{\mathcal{R}G}(F_1) - \pi^u_{\mathcal{R}G}(F_2)|_{\mathcal{R}G,0} \nonumber \\ 
  &\leq&  \frac{1+\epsilon}{\Cr{22}}|F_1 - F_2|_{\mathcal{R}G,0}.
 \end{eqnarray} 
Choose
 $$\Cr{expa}= \frac{1+\epsilon}{\Cr{22}}=\frac{ (1+\epsilon)^2 }{\theta (1-\epsilon)^3 } < 1.$$
Replacing $\mathcal{R}$ by an iteration of it, by (\ref{cinv}) we have (\ref{cinv2}) for every $F_1, F_2 \in \hat{\mathcal{H}}_1$.
\end{proof}

From now on replace  $\mathcal{R}$ by an iteration of it such that  (\ref{cinv2}) holds. Let $\mathcal{H}_0$, $\mathcal{H}_1$ be as in Proposition \ref{trans}. We denote $\mathcal{H}_1= \hat{\mathcal{R}}_G\mathcal{H}_0$. 

\begin{cor} \label{cordon} \label{cor11} Let  $\Cr{22} > 1$ be   as in Proposition \ref{trans}. There exists $\Cl{q1} > 0$ such that the following holds: Given $\Cl{q3} > 0$ and $\delta \in (0,\delta_3]$ there exists $k_2\geq 1$ such that for every  $G \in \Omega$ and $\mathcal{H}_0 \in \mathcal{T}^2_0(G,\delta)$  with base point $F \in W^s_{\delta_3}(G)$  that satisfies  
$$|D^2 \mathcal{H}_0|_{G,0}\leq \Cr{q3}$$
then we have that  
$$\mathcal{H}_k=\hat{\mathcal{R}}_{\mathcal{R}^kG} \mathcal{H}_{k-1} \in \mathcal{T}^2_1(\mathcal{R}^kG,\eta_k)$$
are well defined, with base point $\mathcal{R}^kF \in W^s_{\delta_3}(\mathcal{R}^kG)$, where $\delta_0=\delta$, and 
$$\eta_k~=~\min \{\Cr{22} \eta_{k-1},\delta_3 \}$$ for $k > 0$. Moreover for every $k\geq 1$ 
\begin{equation}\label{k1} |D\mathcal{H}_k|_{(\mathcal{R}G,0), (\mathcal{R}G,1)}\leq \Cr{3} \epsilon\end{equation} 
 and
$$|w+\mathcal{H}_k(w)|_{\mathcal{R}^kG,1} < \Cr{uuu}.$$
Here
$$v_k=\pi^u_{\mathcal{R}^kG}(\mathcal{R}^kF- \mathcal{R}^kG).$$
 Furthermore 
 for every  $k\geq k_2$ we have
\begin{equation}\label{k2}  |D^2 \mathcal{H}_k|_{(\mathcal{R}G,0), (\mathcal{R}G,1)} \leq  \frac{\Cr{q1}}{2}.\end{equation} 
\end{cor}
\begin{proof} By Proposition \ref{trans}.B  it  follows that $\mathcal{H}_k \in \mathcal{T}^2_1(\mathcal{R}^kG,\eta_k)$ for every $k\geq 1$. In particular if $k \geq k_0= \min\{k \in \mathbb{N} \ s.t. \ \theta_1^k\delta > \delta_3\}+1$ we have that $\mathcal{H}_k \in \mathcal{T}^2_1(\mathcal{R}^kG,\delta_3)$.  By (\ref{est55}) we have
$$ |D^2 \mathcal{H}_{k}|_{(\mathcal{R}G,0), (\mathcal{R}G,1)} \leq  \Cr{uu} + {\Cr{1}}  |D^2 \mathcal{H}_{k-1}|_{G,0},$$
for every $k$, so
$$ |D^2 \mathcal{H}_{k}|_{(\mathcal{R}G,0), (\mathcal{R}G,1)} \leq  \sum_{j=0}^{k-1} \Cr{uu}  {\Cr{1}}^j +  {\Cr{1}}^k  |D^2 \mathcal{H}_{0}|_{G,0}\leq  \frac{\Cr{uu}}{1- {\Cr{1}}}  +  {\Cr{1}}^k \Cr{q3},$$
Choose $\Cr{q1}$ such that
$$\frac{\Cr{uu}}{1- {\Cr{1}}} < \frac{\Cr{q1}}{2}$$
and $k_1 \geq k_0$ satisfying 
$$ \frac{\Cr{uu}}{1- {\Cr{1}}}   + {\Cr{1}}^{k_1} \Cr{q3} <   \frac{\Cr{q1}}{2}.$$
Then (\ref{k2}) holds for every $k\geq k_1$.   \end{proof}

Let $k_2$ given by Corollary  \ref{cordon} when we choose  $ \Cr{q3}=\Cr{q1}$. From now on we  replace $\mathcal{R}$ by $\mathcal{R}^{k_2}$.

\section{Lower bound to the measure of  parameters outside $W^s_{\delta_3}(\Omega)$.}\label{lower_bound}

\begin{figure}
%if you want to compile with psfrag replace pdf by eps file and uncomment \psfrag lines. 
%\psfrag{g}[][][0.9]{$G$}
%\psfrag{gk}[][][0.9]{$G_k$}
%\psfrag{A}[][][0.9]{$A$}
%\psfrag{es}{$E^h_G$}
%\psfrag{eu}{$E^u_G$}
%\psfrag{esk}{$E^h_{G_k}$}
%\psfrag{euk}{$E^u_{G_k}$}
%\psfrag{Bk}[][][0.7][-9]{$\mathbb{B}^u_{G_k}(v_k,\delta)$}
%\psfrag{v}[][][0.9]{$v_k$}
%\psfrag{h}[][][0.6]{$Graph \ \mathcal{H}_k$}
%\psfrag{ht}[][][0.7]{$Graph \ \tilde{\mathcal{H}}_k$}
%\includegraphics[scale=0.8]{conv3.eps}
\includegraphics[scale=0.35]{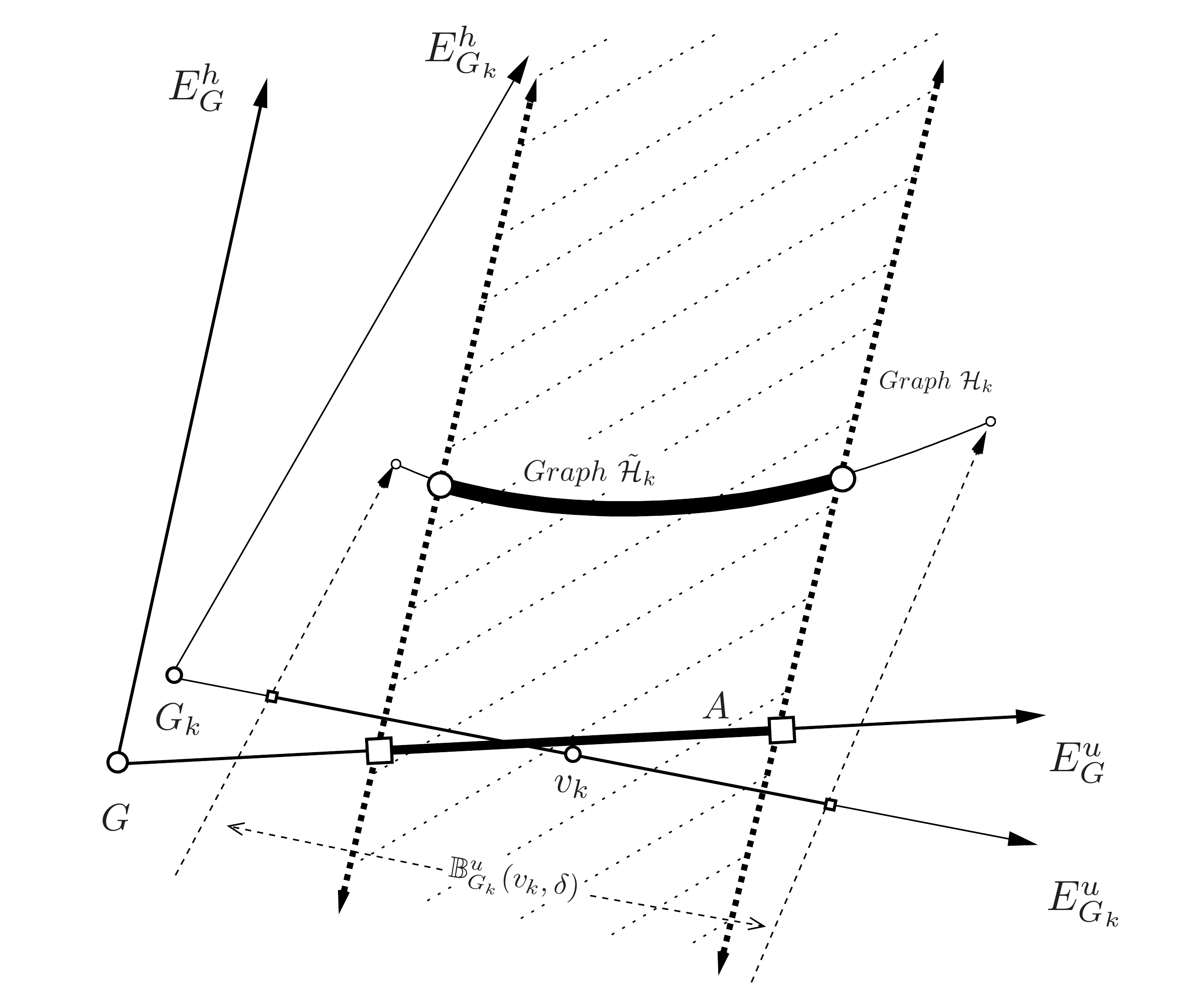}
\caption{Convergence of transversal families. }
\end{figure}

We say that  a sequence of $C^j$ functions $\mathcal{H}_k \in \mathcal{T}^j_0(G_k,\delta)$, $\delta \in (0,\delta_3)$, with base point $F_k$ converges to a $C^j$ function $\mathcal{H} \in \mathcal{T}^j_0(G,\delta)$ with base point $F$  if 
$$\lim_k G_k=G, \lim_k F_k=F$$
and for  every open set $A$ compactly contained in $\mathbb{B}^u_G(v_\infty,\delta)$ there is $k_0$ such that for $k > k_0$ the 
set 
$$\{ w+\mathcal{H}_k(w)\colon \  w \in   \mathbb{B}^u_{G_k}(v_k,\delta)  \}\cap \{ u+v+G\colon \ u \in A, \ v \in E^h_G   \}$$
is the graph (See Figure 2 ) of a $C^j$ function 
$$\tilde{\mathcal{H}}_k\colon A \rightarrow E^h_G+G$$
and 
$$\lim_k |\tilde{\mathcal{H}}_k-\mathcal{H}|_{C^j(A)}=0.$$
Here
$$v_k = \pi^u_{G_k}(F_k-G_k) \ and \ v_\infty = \pi^u_{G}(F-G).$$

\begin{prop} \label{comp33} Let    $\mathcal{H}_k \in \mathcal{T}^2_1(G_k,\delta)$, $\delta \in (0,\delta_3)$,  be a sequence of $C^2$ functions, for some $G_k\in \Omega$, base point $F_k$ and satisfying  
\begin{equation}\label{upper2}   |D \mathcal{H}_k|_{(G_k,0), (G_k,1)} \leq \epsilon,  |D^2 \mathcal{H}_k|_{(G_k,0), (G_k,1)}\leq \Cr{q1}.\end{equation}
Moreover assume there exists $\Cl{limitado}$ such  that 
$$|w+\mathcal{H}_k(w)|_{G_k,1}\leq \Cr{limitado},$$
 for every $w \in \mathbb{B}^u_{G_k}(v_k,\delta)$,    with
$$v_k = \pi^u_{G_k}(F_k-G_k), $$ and $k\geq 0$. Then there exists a subsequence $\mathcal{H}_{k_i}$ that converges to a $C^{1+Lip}$ function $\mathcal{H} \in \mathcal{T}^1_0(G,\delta)$, for some $G \in \Omega$.
\end{prop}
\begin{proof}Since $\Omega$ is compact, we can assume that the sequence $G_k$ converges on $\mathcal{B}_0$  to some $G \in \Omega$.   We claim that for $k$ large enough the map 
$$\pi_k\colon \mathbb{B}^u_{G_k}(v_k,\delta)\rightarrow E^u_G$$
defined as
\begin{equation} \label{proj} \pi_k(w)= \pi^u_G ( w+\mathcal{H}_k(w)-G).\end{equation}
is a homeomorphism on its image. It is enough to show that this map is injective. Indeed given $\gamma >0$ then for $k$ large enough we have
$$|\pi^u_{G_k}-\pi^u_{G}|_{0}, \ |\pi^h_{G_k}-\pi^h_{G}|_0 < \frac{\gamma}{\Cr{comp}} .$$
Choose $\gamma$ such that $\gamma(1+\epsilon) < 1$. Note that 
\begin{eqnarray}\label{little} \pi^u_G ( w+\mathcal{H}_k(w) - G)&=& \pi^u_{G_k} ( w+\mathcal{H}_k(w))-\pi^u_G(G) + (\pi^u_{G}-\pi^u_{G_k}) ( w+\mathcal{H}_k(w)) \nonumber \\
&=&   w + \pi^u_G(G_k-G)+ (\pi^u_{G}-\pi^u_{G_k}) ( w+\mathcal{H}_k(w)),   \end{eqnarray}
so if 
$$ \pi^u_G ( w+\mathcal{H}_k(w)-G)= \pi^u_G ( w'+\mathcal{H}_k(w')-G)$$
we would have 
\begin{eqnarray} |w-w'|_{G_k,0} &=& |(\pi^u_{G}-\pi^u_{G_k})(w-w'+\mathcal{H}_k(w)- \mathcal{H}_k(w'))|_{G_k,0}  \nonumber \\
&\leq&  \gamma(1+\epsilon)  |w-w'|_{G_k,0},
\end{eqnarray}
which implies $w=w'$. So the map defined in  (\ref{proj}) is injective.  Since  $F_k$ is a bounded sequence   in $\mathcal{B}_1$, we can assume that $F_k$ converges to some $F \in \mathcal{B}_0$. So
$$v_\infty=\lim_k v_k  = \pi^u_G(F-G).$$
It is easy to see  that $F \in W^s_{\delta_3}(G)$. 
Define 
$$\pi_k\colon \mathbb{B}^u_{G_k}(v_k,\delta)\rightarrow E^u_G$$
as
$$\pi_k(w)= \pi^u_G ( w+\mathcal{H}_k(w)-G).$$
Then by (\ref{little}) 
\begin{equation} \label{quasi} (1-\tilde{\gamma}_k (1+\epsilon)) |w-w'|_{G_k,0} \leq |\pi_k(w)-\pi_k(w')|_{G,0}\leq (1+\tilde{\gamma}_k (1+\epsilon))  |w-w'|_{G_k,0},\end{equation}   
where
$$\lim_k \tilde{\gamma}_k=0.$$
Since  $\pi_k$ is a homeomorphism on its image, (\ref{quasi})  implies that for every $\delta' < \delta$ there exists $k_0$ such that if $k > k_0$ then 
$$\mathbb{B}^u_G(v_\infty,\delta')\subset  \pi_k(\mathbb{B}^u_{G_k}(v_k,\delta)),$$
so for every convex open set $A$ compactly contained in $\mathbb{B}^u_G(v_\infty,\delta)$ there is $k_0$ such that if $k > k_0$ then
$$A \subset  \pi_k(\mathbb{B}^u_{G_k}(v_k,\delta)).$$
Define
$$ \tilde{\mathcal{H}}_k\colon A \rightarrow E^h_G+G  $$
as
$$ \tilde{\mathcal{H}}_k(u)=     \pi^h_G(   (Id + \mathcal{H}_k)\circ  (\pi_k)^{-1}(u) -G).$$ 
\ \\
\noindent {\bf Claim.} {\it  There exists $\Cl{48}$ such that 
$$|D^b\pi^h_G\circ (Id + \mathcal{H}_k)\circ  (\pi_k)^{-1}(u)|_{(G,0), (G,1)} \leq \Cr{48}$$
for every $u \in A$ and $k\geq k_0$, $b \in \{0,1,2\}$, that is,
$$|D^b \tilde{\mathcal{H}}_k|_{(G,0),(G,1)} \leq \Cr{48},$$
for every $b \in \{0,1,2\}$.} \\

\noindent Indeed, note that 
$$|\pi^h_G\circ (Id + \mathcal{H}_k)\circ  (\pi_k)^{-1}(u)|_{G,1}\leq \Cr{c55} \Cr{limitado}.$$
Moreover 
$$D_w\pi_k\cdot v = \pi^u_G ( v+D_w\mathcal{H}_k\cdot v).$$
$$D^2_w\pi_k\cdot v^2 = \pi^u_G (D^2_w \mathcal{H}_k\cdot v^2).$$
and
$$D_w\pi_k^{-1}\cdot v = (D_{\pi_k^{-1}(w)}\pi_k)^{-1} \cdot v, $$
$$D_w^2\pi_k^{-1}\cdot v^2 = - D_w\pi_k^{-1} \cdot  \pi^u_G \cdot D^2_{\pi_k^{-1}(w)} \mathcal{H}_k \cdot (D_{w}\pi_k^{-1} \cdot v )^2 $$
If $k$ is large enough  then
$$|D_w\pi_k\cdot v|_{G,0} \geq  \Cr{comp}^{-2}(1-\epsilon)^2  |v|_{G_k,0},$$
so
$$|D_w\pi_k^{-1}|_{(G,0), (G_k,0)}\leq \Cr{comp}^2 (1-\epsilon)^{-2}.$$
and
$$|D_w^2\pi_k^{-1}|_{(G,0),(G_k,0)} \leq \Cr{comp}^8 (1-\epsilon)^{-6}  \Cr{q1}.$$
So using  (v), (vii) and (\ref{upper2}) 
\begin{eqnarray} &|D \pi^h_G\circ (Id + \mathcal{H}_k)\circ  \pi_k^{-1}(u)|_{(G,0), (G,1)} \nonumber \\
& \leq  |\pi^h_G|_{(G_k,1), (G,1)}  |Id + D_{\pi_k^{-1}(u)}\mathcal{H}_k|_{(G_k,0),(G_k,1)} | D\pi_k^{-1}(u)|_{(G,0), (G_k,0)} \nonumber \\
&\leq \Cr{c55}  (\Cr{comp44} + \epsilon ) \Cr{comp}^2 (1-\epsilon)^{-2}.
 \end{eqnarray} 
 and
 \begin{eqnarray} & &|D^2 \pi^h_G\circ (Id + \mathcal{H}_k)\circ  \pi_k^{-1}(u)|_{(G,0), (G,1)} \nonumber \\
&\leq& | \pi^h_G \cdot D^2_u  \pi_k^{-1}|_{(G,0), (G,1)} +  |\pi^h_G \cdot D^2_{\pi_k^{-1}(u)} \mathcal{H}_k \cdot (D_u  \pi_k^{-1})^2|_{(G,0), (G,1)} \nonumber \\
&+& |\pi^h_G \cdot D_{\pi_k^{-1}(u)} \mathcal{H}_k \cdot D_u^2  \pi_k^{-1}|_{(G,0), (G,1)}  \nonumber \\
&\leq&  \Cr{c55}  \Cr{comp44} \Cr{comp}^8 (1-\epsilon)^{-6}  \Cr{q1}  + \Cr{c55} \Cr{q1} \Cr{comp}^4 (1-\epsilon)^{-4} + \Cr{c55}   \epsilon \Cr{comp}^8 (1-\epsilon)^{-6}  \Cr{q1}.
 \end{eqnarray} 
This proves the claim.  In particular, since $A$ is convex,  the maps
\begin{equation} \label{ascoli} u\mapsto D^b \tilde{\mathcal{H}}_k(u) \ are \ uniformly \ Lipschitz \ on  \ A,\end{equation}
for $b \in \{0,1\}$, the set 
$$\{\tilde{\mathcal{H}}_k(u)\colon \ u \in A, \ k\geq k_1   \}$$
is a   relatively compact subset of $\mathcal{B}_0$, and 
$$\{D \tilde{\mathcal{H}}_k(u)\colon \ u \in A, \ k\geq k_1   \}$$
is a   relatively compact subset of the space of all continuous operators of $\mathcal{B}_0$.  Let $\{ u_i\}_{i\in \mathbb{N}}$ be a dense subset of $A$. Then using the Cantor's diagonal argument  one can find a subsequence $k_j$ such that the limits
$$\lim_j  \ D^b \tilde{\mathcal{H}}_k (u_i),$$
exists for every $i \in \mathbb{N}$.  Then (\ref{ascoli}) implies that 
$$\mathcal{H}^b(u)= \lim_j  \  D^b \tilde{\mathcal{H}}_k(u),$$
exists for every $u \in A$, $b\in \{0,1\}$. This convergence is uniform on $u\in A$. One can easily conclude that
$$D\mathcal{H}^0=\mathcal{H}^1$$
So  $\mathcal{H}=\mathcal{H}^0$ is $C^{1+Lip}$ and $\tilde{\mathcal{H}}_k$ converges on $A$ to $\mathcal{H}$ in $C^1$ topology.  Since there exists an exhaustion of $\mathbb{B}^u_G(v_\infty,\delta)$ by convex, open and relatively compact sets $A_k$, we can use Cantor's diagonal argument once again to find a $C^{1+Lip}$ map $\mathcal{H} \in \mathcal{T}^1_0(G,\delta)$ such that $\mathcal{H}_k$ converges to $\mathcal{H}$ in $C^1$ topology. 
\end{proof}
For every $\mathcal{H} \in \mathcal{T}^1_0(G,\delta)$ we can define a borelian measure $m_{\mathcal{H}}$ on $\hat{\mathcal{H}}$ in the following way.  The measure $m_{\mathcal{H}}$ is the $n$-dimensional {\it Hausdorff measure } on $\hat{\mathcal{H}}$ with respect to the metric induced in $\hat{\mathcal{H}}$ by the norm of $|\cdot |_{G,0}$.  

Fix some $\Cl{iso} >0 $. For each $G \in \Omega$, let $\gamma_G > 0$ be such that for every $\tilde{G} \in \mathbb{B}(G,\gamma_G)\cap \Omega $ the map 
$$\pi_{\tilde{G}}^u\colon E^u_G\mapsto E^u_{\tilde{G}}$$
is a linear isomorphism whose norm and the norm of its inverse is  bounded by some constant $\Cr{iso}$, considering the norm   $|\cdot|_{G,0}$ on $E^u_G$ and the norm $|\cdot|_{\tilde{G},0}$  on $E^u_{\tilde{G}}$. This is possible since $G\mapsto \pi^u_G$ is continuous with respect to  the $\mathcal{B}_0$ norm.  By the compactness of $\Omega$ there is a finite set $G_1^\star, \dots, G_m^\star \in \Omega$ such that
$$\Omega \subset \cup_i \mathbb{B}(G_i^\star,\gamma_{G_i}).$$
 Fix a basis $v^i_1,\dots v^i_n$ for $E^u_{G_i^\star}$.  
For every $$G\in \mathbb{B}(G_i^\star,\gamma_{G_i^\star})\cap\Omega $$ we have a basis $v^{G,i}_1,\dots, v^{G,i}_n$ for $E^u_{G}$ given by $v^{G,i}_j=\pi^u_{G}(v^{i}_j)$.  Let $| \cdot  |_{G,G^\star_i}$ be the norm on $E^u_G$  that turns $E^u_G$ into a Hilbert space and $\mathcal{S}_{G,i} = (v^{G,i}_1,\dots, v^{G,i}_n)$ into a orthonormal basis of it. 
There is $\Cl{linear} > 0$  such that  
$$\frac{1}{ \Cr{linear}} | v |_{G_i^\star,G_i^\star}\leq | v |_{G_i^\star,0}\leq \Cr{linear} | v |_{G_i^\star,G_i^\star}$$
for every $i$. This implies that  
\begin{equation}\label{compn2} \frac{1}{ \Cr{iso} \Cr{linear}} | v |_{G,G_i^\star}  \leq | v |_{G,0} \leq \Cr{iso} \Cr{linear} | v |_{G,G_i^\star}\end{equation}
Together with (\ref{comp78}), we conclude that the  norms $|\cdot|_{G,G_i^\star}$, $|\cdot|_{G,0}$ and $|\cdot|_0$ are not only equivalent on $E^u_G$ (which  is obvious, once $E^u_G$ has finite dimension), but  also that there is a universal upper bound to the norm of the identity maps $Id\colon E^u_G \mapsto E^u_G$ that holds considering  every  $G \in \Omega$ and every one of  these three norms on its domain and range.

Let $m_{G,i}$ be the Lebesgue measure  on $E^u_{G}$ such that 
\begin{equation}\label{box} m_{G,i}\{v \in E^u_{G}\colon \ v = \sum_j \alpha_j  v^{G,i}_j, \ with \  \alpha_j \in [0,1]\} =1.\end{equation}
Note that the quotient of the measure $m_{G,i}$ by the $n$-dimensional Hausdorff measure induced by the norm  $| \cdot |_{G,G_i^\star}$ is a constant that  depends only on the dimension $n$ of $E^u$.  Of course by the uniqueness of the Haar measure on locally compact topological groups,  if $m_{G}$ is the $n$-dimensional Hausdorff measure induced by the norm $|\cdot|_{G,0}$ on $E^u_{G}$ and  $G \in \mathbb{B}(G_i^\star,\gamma_{G_i^\star})$ then there exists  
$K_{G,i} > 0$ such that $m_{G,i}= K_{G,i} m_{G}$. From (\ref{compn2}) it easily follows that 
\begin{equation}\label{compk} \frac{\sigma(n)}{\Cr{iso}^{n} \Cr{linear}^{n}}  \leq K_{G,i}\leq  \sigma(n) \Cr{iso}^n \Cr{linear}^n.\end{equation}
where $\sigma(n)$ is a constant that depends only on $n$. Finally given some $\mathcal{H} \in \mathcal{T}_0^k(G,\delta)$ with base point $F$  then 
$$\Pi^u_{G}\colon \hat{\mathcal{H}}\mapsto \mathbb{B}^u_{G}(\pi^u_{G}(F-G),\delta)$$
given by $\Pi^u_{G}(y)= \pi^u_{G}(y-G)$ is a bilipchitz map and the Lipchitz constant of $\Pi^u_{G}$ and its inverse is at most $1+\varepsilon$, considering  the metric induced by $|\cdot|_{G,0}$ on $\hat{\mathcal{H}}$ and $\mathbb{B}^u_{G}(\pi^u_G(F-G),\delta)$. This implies that for every $A \subset  \hat{\mathcal{H}}$ we have
\begin{equation}\label{compkk} (1+\epsilon)^{-n} \leq \frac{m_{\mathcal{H}}(A)}{m_{G}(\Pi^u_{G}(A))} \leq (1+\epsilon)^n. \end{equation}

Note also  that if $m_{\mathcal{H},\mathcal{B}_0}$ is the $n$-dimensional Hausdorff measure on $\hat{\mathcal{H}}$  induced by the norm $|\cdot|_{0}$ then by (\ref{comp78}) we have that 
$$  \frac{1}{\Cr{comp}^n} \leq    \frac{m_{\mathcal{H}}(A)}{m_{\mathcal{H},\mathcal{B}_0}(A)}\leq \Cr{comp}^n.$$
for every borelian set $A \subset  \hat{\mathcal{H}}$.

\begin{lem} \label{lower}Given $\delta > 0$, there exists $\Cl{lowerballs}(\delta) > 0$ and  $\Cl{lowerballs2}(\delta) > 0$ with the following property.  Suppose that $\mathcal{H} \in \mathcal{T}^1_0(G, \tilde{\delta})$, with base point $\tilde{F}$.
\begin{itemize}
\item[A.] If $F \in \hat{\mathcal{H}}$  satisfies 
 $$\mathbb{B}^u_{G}(\pi^u_G(F-G) , \Cr{comp} \delta) \subset \mathbb{B}^u_{G}(\pi^u_G(\tilde{F}-G) , \tilde{\delta}).$$
then
$$m_{\mathcal{H}}\{ F_1 \in \hat{\mathcal{H}}\colon \ |F_1-F|_{0} \leq  \delta \} \leq \Cr{lowerballs}(\delta).$$
\item[B.] If $F \in \hat{\mathcal{H}}$  satisfies 
 $$\mathbb{B}^u_{G}(\pi^u_G(F-G) , \frac{\delta}{\Cr{comp}(1+\varepsilon)}) \subset \mathbb{B}^u_{G}(\pi^u_G(\tilde{F}-G) , \tilde{\delta}).$$
then
$$m_{\mathcal{H}}\{ F_1 \in \hat{\mathcal{H}}\colon \ |F_1-F|_{0} < \delta\} \geq \Cr{lowerballs2}(\delta).$$
\end{itemize}
\end{lem}
\begin{proof}[Proof of A] Note that if $F_1 \in \hat{\mathcal{H}}$ satisfies 
$$|F_1-F|_{0} < \delta$$
then 
$$ \pi^u_G(F_1-G)  \in \mathbb{B}^u_{G}(\pi^u_G(F-G) , \Cr{comp} \delta ).$$
Denote $\tilde{\gamma}= \Cr{comp}\delta$ and $w_0=\pi^u_G(F-G)$. 
We have that $G\in \mathbb{B}(G_{i}^\star,\gamma_{G_{i}^\star})$, for some $i$. So 
$$\mathbb{B}^u_{G_k}(w_0 ,\tilde{\gamma}) \subset \{w_0+w \in E^u_{G_k}\colon \ w = \sum_j \alpha_j  v^{G,i_k}_j, \ with \  |\alpha_j|\leq \Cr{iso} \Cr{linear}  \tilde{\gamma}\} $$
so by (\ref{box}) and (\ref{compk}) we obtain 
\begin{eqnarray}
m_{\mathcal{H}}\{ F_1 \in \hat{\mathcal{H}}\colon \ |F_1-F|_{0} < \delta\} &\leq& m_{\mathcal{H}}\{v+\mathcal{H}(v)\colon \ v \in  \mathbb{B}^u_{G}(w_0 ,\tilde{\gamma}) \} \nonumber \\ &\leq& (1+\varepsilon)^n m_{G}(\mathbb{B}^u_{G}(w_0 ,\tilde{\gamma})) \nonumber \\
&\leq& \frac{\Cr{iso}^n \Cr{linear}^n (1+\varepsilon)^n}{ \sigma(n)} m_{G,i}(\mathbb{B}^u_{G}(w_0,\tilde{\gamma})) \nonumber \\
&\leq& \frac{ \Cr{iso}^n \Cr{linear}^n (1+\varepsilon)^n}{ \sigma(n) }   (2\Cr{iso} \Cr{linear} \tilde{\gamma})^n.  
\end{eqnarray}
\end{proof}
\begin{proof}[Proof of B] If 
$$v \in \mathbb{B}^u_{G}(\pi^u_G(F-G) , \frac{\delta}{2\Cr{comp}(1+\varepsilon)})$$
then $$v + \mathcal{H}(v) \in \{ F_1 \in \hat{\mathcal{H}}\colon \ |F_1-F|_{0} < \delta\}.$$
Denote $\tilde{\gamma}= \frac{\delta}{\Cr{comp}(1+\varepsilon)}$ and $w_0=\pi^u_G(F-G)$. 
We have that $G\in \mathbb{B}(G_{i}^\star,\gamma_{G_{i}^\star})$, for some $i$. So 
$$\{w_0+w \in E^u_{G_k}\colon \ w = \sum_j \alpha_j  v^{G,i_k}_j, \ with \  \alpha_j\in  [0,\frac{\tilde{\gamma}}{ \Cr{iso} \Cr{linear} \sqrt{n}}]\}  \subset  \mathbb{B}^u_{G_k}(w_0 ,\tilde{\gamma})$$
so by (\ref{box}) and (\ref{compk}) we obtain 
\begin{eqnarray}
m_{\mathcal{H}}\{ F_1 \in \hat{\mathcal{H}}\colon \ |F_1-F|_{0} < \delta \} &\geq& m_{\mathcal{H}}\{v+\mathcal{H}(v)\colon \ v \in  \mathbb{B}^u_{G}(w_0 ,\tilde{\gamma}) \} \nonumber \\ &\geq& \frac{1}{(1+\varepsilon)^n} m_{G}(\mathbb{B}^u_{G}(w_0 ,\tilde{\gamma})) \nonumber \\
&\geq& \frac{1}{ \sigma(n) \Cr{iso}^n \Cr{linear}^n (1+\varepsilon)^n} m_{G,i}(\mathbb{B}^u_{G}(w_0,\tilde{\gamma})) \nonumber \\
&\geq& \frac{1}{ \sigma(n) \Cr{iso}^n \Cr{linear}^n (1+\varepsilon)^n} \Big( \frac{\tilde{\gamma}}{ \Cr{iso} \Cr{linear} \sqrt{n}} \Big)^n.  
\end{eqnarray}
\end{proof}

\begin{prop}\label{measureh} Let $\Cr{q1} > 0$ as in Corollary \ref{cordon}. If  $\delta \in (0, \delta_3)$  and $\Cl[c]{ball} \in (0,1) $ then there exists $\Cl{medhip}= \Cr{medhip}(\delta, \Cr{ball}) > 0$   such that the following holds. Let $\mathcal{H} \in \mathcal{T}^2_1(G, \delta)$  be a 
$C^2$ function  
$$\mathcal{H}\colon \mathbb{B}^u_G(v_0,\delta) \rightarrow E^h_G+G$$
with $G\in \Omega$, base point $F$  and  satisfying  
\begin{equation}\label{upper}   |D \mathcal{H}|_{(G,1),(G,0)} \leq \epsilon,  |D^2 \mathcal{H}|_{(G,1),(G,0)}  \leq \Cr{q1}.\end{equation}
Moreover assume there exists $\Cl{limitado2}$ such  that 
$$|w+\mathcal{H}_k(w)|_{1}\leq \Cr{limitado2},$$
 for every $w \in \mathbb{B}^u_{G}(v_0,\delta)$. Then we have that
$$m_{\mathcal{H}}(\hat{\mathcal{H}}\cap W^s_{\delta_3}(\Omega)^c\cap \mathbb{B}(F,\Cr{ball}\delta )) \geq \Cr{medhip}.$$
\end{prop}
\begin{proof}Suppose that there exists  a sequence of maps  $G_k \in \Omega$ and $C^2$ functions  $\mathcal{H}_k$ 
$$\mathcal{H}_k\colon \mathbb{B}^u_{G_k}(\pi^u_{G_k}(F_k-G_k),\delta) \rightarrow E^s_{G_k}+G_k$$
such that $\mathcal{H}_k \in \mathcal{T}^2_1(G_k,\delta)$, with base point $F_k$  and moreover
\begin{equation}\label{zerom} \lim_k m_{\mathcal{H}_k}(\hat{\mathcal{H}}\cap W^s_{\delta_3}(\Omega)^c\cap \mathbb{B}(F_k,\Cr{ball}\delta ))=0.\end{equation}
Since $\Omega$ is compact we can replace $G_k$ by a subsequence such  that 
$$\lim_k G_k= G \in \Omega.$$
Since $\mathcal{H}_k$ satisfies (\ref{upper}), by Proposition \ref{comp33}  we can assume without loss of generality that the functions $\mathcal{H}_k$ converges in $C^1$ topology  to a $C^{1+Lip}$ function $\mathcal{H} \in \mathcal{T}^1_0(G,\delta)$ with a base point $F$. Moreover $F_k$ converges to $F$ in $\mathcal{B}_0$. By the transversal empty interior assumption   the transversal family $u\mapsto u +\mathcal{H}(u)$ has a parameter  
$$u_\infty\in \mathbb{B}^u_{G}(\pi^u_{G}(F-G),\frac{1}{3}\frac{\Cr{ball} \delta}{\Cr{comp}(1+\epsilon)})$$
 such that $ u_\infty +\mathcal{H}(u_\infty) \in W^s_{\delta_3}(\Omega)^c\cap \mathbb{B}(F,\Cr{ball} \delta/3)$.  Since  $W^s_{\delta_3}(\Omega)$ is a closed set there is $\gamma> 0$ be  such that  $\mathbb{B}(u_\infty +\mathcal{H}(u_\infty),\gamma) \subset  W^s_{\delta_3}(\Omega)^c\cap \mathbb{B}(F,\Cr{ball} \delta/3)$.

 Since the families  $u\mapsto u +\mathcal{H}_k(u)$ converges to this family in $C^1$ topology we can easily conclude that for large $k$  there exists $$u_k \in \mathbb{B}^u_{G_k}(\pi^u_{G_k}(F_k-G_k),\frac{2}{3} \frac{\Cr{ball} \delta}{\Cr{comp}(1+\epsilon)} )$$  such that 
 $$| u_\infty +\mathcal{H}(u_\infty) - u_k -\mathcal{H}_k(u_k)|_{0} < \frac{\gamma}{2}.$$
 Let $$\tilde{\gamma}= \min \{ \frac{1}{3}\frac{\Cr{ball} \delta}{\Cr{comp}(1+\epsilon)}, \frac{\gamma}{2\Cr{comp}(1+\varepsilon)}\}.$$
If  $$v  \in \mathbb{B}^u_{G_k}(u_k,\tilde{\gamma}) \subset \mathbb{B}^u_{G_k}(\pi^u_{G_k}(F_k-G_k), \frac{\Cr{ball} \delta}{\Cr{comp}(1+\epsilon)} )$$ then $$v+\mathcal{H}_k(v) \in \mathbb{B}(u_\infty +\mathcal{H}(u_\infty),\gamma),$$ so $v+\mathcal{H}_k(v) \in  W^s_{\delta_3}(\Omega)^c\cap \mathbb{B}(F_k,\Cr{ball}\delta)$.  We have that $G_k\in \mathbb{B}(G_{i_k}^\star,\gamma_{G_{i_k}^\star})$, for some $i_k$. Note that (\ref{compn2}) implies 
$$\{u_k+v \in E^u_{G_k}\colon \ v = \sum_j \alpha_j  v^{G_k,i_k}_j, \ with \  \alpha_j \in [0,\frac{\tilde{\gamma}}{ \Cr{iso} \Cr{linear} \sqrt{n}}]\}  \subset  \mathbb{B}^u_{G_k}(u_k,\tilde{\gamma})$$
so by (\ref{box}) and (\ref{compk}) we obtain 
\begin{eqnarray}
m_{\mathcal{H}_k}( \hat{\mathcal{H}} \cap W^s_{\delta_3}(\Omega)^c \cap \mathbb{B}(F,\Cr{ball}\delta)) &\geq& m_{\mathcal{H}_k}\{v+\mathcal{H}_k(v)\colon \ v \in  \mathbb{B}^u_{G_k}(u_k,\tilde{\gamma}) \} \nonumber \\ &\geq& \frac{1}{(1+\varepsilon)^n} m_{G_k}(\mathbb{B}^u_{G_k}(u_k,\tilde{\gamma})) \nonumber \\
&\geq& \frac{1}{ \sigma(n) \Cr{iso}^n \Cr{linear}^n (1+\varepsilon)^n} m_{G,i_k}(\mathbb{B}^u_{G_k}(u_k,\tilde{\gamma})) \nonumber \\
&\geq& \frac{1}{ \sigma(n) \Cr{iso}^n \Cr{linear}^n (1+\varepsilon)^n} \Big( \frac{\tilde{\gamma}}{ \Cr{iso} \Cr{linear} \sqrt{n}} \Big)^n.  
\end{eqnarray}
So 
$$\limsup_k  m_{\mathcal{H}_k}( W^s_{\delta_3}(\Omega)^c\cap \mathbb{B}(F,\Cr{ball}\delta)) > 0,$$
which is a contradiction with (\ref{zerom}).
\end{proof} 

\section{Dynamical balls in transversal families.}
\label{dballs}

Given $F \in W^s_{\delta_3}(\Omega)$, let $G_F$ be an  element of $\Omega$ such that $F \in W^s_{\delta_3}(G_F)$. Choose $\delta_4 < \delta_3$  in such way that if $G_1, G_2 \in \Omega$ and $|G_1 - G_2|_{G_1}\leq 3 \Cr{comp} \delta_4$ then 
\begin{equation} \label{comp77}   \frac{1}{2} |v|_{G_2,0}    \leq |v|_{G_1,0}  \leq 2 |v|_{G_2,0} \end{equation} 

Fix $G \in \Omega$. In this section $\mathcal{H}$ is a $C^2$  function
$$\mathcal{H}\colon W \rightarrow E^h_G+G,$$
where $W \subset E^u_G$ is an convex open set. We can choose   $\delta_5  <  \delta_4$ such that  if 
\begin{itemize}
\item[A1.] $|D \mathcal{H}|_{G,0} \leq \Cr{3}  \epsilon$ and 
\item[A2.]  $\inf_{x \in W} |x+\mathcal{H}-G|_0 + (1+\epsilon) diam_G W < 2\Cr{comp}\delta_5$,
\end{itemize}
then for every 
$$F \in \hat{\mathcal{H}}\cap W^s_{\delta_5}(\Omega),$$
and for every  $\delta \in (0,  \delta_5/(3 \Cr{comp}))$  satisfying 
$$\mathbb{B}^u_G( \pi^u_G(F-G)  ,3 \Cr{comp} \delta) \subset W$$
we have that 
$$\{ w+\mathcal{H}(w)\colon \  w \in   W \}\cap \{ u+v+G_F\colon \ u \in \mathbb{B}^u_{G_F}(\pi^u_{G_F}(F-G_F),2 \Cr{comp} \delta), \ v \in E^s_{G_F}   \}$$
is the graph $\hat{\mathcal{H}}_F$ of a $C^2$ function 
$$\mathcal{H}_F \in \mathcal{T}(G_F, 2 \Cr{comp} \delta)$$
with base point $F$ and 
$$|D \mathcal{H}_F|_{G_F,0}\leq  \epsilon$$
and  additionally if we assume 
\begin{itemize}
\item[A3.]  $|D^2 \mathcal{H}|_{G,0} \leq  \Cr{q1}/2$.
\end{itemize}
then 
$$|D^2 \mathcal{H}|_{G_F,0}\leq \Cr{q1}.$$

The goal of this section is to show 
\begin{prop}\label{zerohip}  Let  $\mathcal{H}$ be a $C^2$ function  satisfying $A1$-$A3$. Then
\begin{equation}m_{\mathcal{H}}(W^s_{\delta_5}(\Omega)\cap\hat{\mathcal{H}} )=0.  \end{equation}
\end{prop}

\subsection{Dynamical balls}
Let $\mathcal{H}$ be as in Proposition \ref{zerohip}. Define the {\bf dynamical ball}
$$B_{\mathcal{H}} (F,\delta,k) =\{F_1 \in   \hat{\mathcal{H}}\colon   |\mathcal{R}^iF_1 - \mathcal{R}^iF|_{0} < \delta, \text { for every } 0\leq i \leq k \}.$$
We can choose $\delta < \delta_5$ small enough such that 
\begin{equation} \label{dynball} B_{\mathcal{H}} (F,\delta,0) \subset \hat{\mathcal{H}}_F,\end{equation}
Consider the map $\phi_u$ as in  the proof of Propostition \ref{trans}
$$x\in E^u_G \mapsto  \phi_u(x) = \pi^u_{\mathcal{R}G} (\mathcal{R}(x+\mathcal{H}(x) + G)-\mathcal{R}G) \in  E^u_{\mathcal{R} G}.$$

\begin{lem}\label{stepmeasure}  Let $\Cr{q1}$ be as in Corollary  \ref{cordon}.  There exists $\Cr{estlip34}$ such that  for every  $G \in \Omega$ and $\mathcal{H} \in \mathcal{T}^2_0(G,2\Cr{comp} \delta)$  with base point $F=v_0+\mathcal{H}(v_0) \in W^s_{\delta_5}(G)$  that satisfies  
$$|D^2 \mathcal{H}|_{G,0} \leq \Cr{q1}$$
then the following holds. Consider the map $\phi_u$ as in  the proof of Propostition \ref{trans} 
$$\phi_u\colon \mathbb{B}^u_G(v_0,2\Cr{comp} \delta)   \mapsto E^u_{\mathcal{R} G}$$
defined by 
$$\phi_u(x) = \pi^u_{\mathcal{R}G} (\mathcal{R}(x+\mathcal{H}(x))-\mathcal{R}G).$$
If $S_1$ and $S_2$ are borelian subsets of $E^u_G$, with 
$$S_1\subset S_2 \subset  \mathbb{B}^u_G(v_0,2\Cr{comp} \delta)$$
 then 
\begin{equation}\label{dm}
\frac{m_{G,i}(S_1)}{m_{G,i}(S_2)} e^{-2\Cr{estlip34} diam_{G,0}{S_2}}   \leq  \frac{m_{\mathcal{R}G,j}(\phi_u(S_1))}{m_{\mathcal{R}G,j}(\phi_u(S_2))} \leq    \frac{m_{G,i}(S_1)}{m_{G,i}(S_2)} e^{2\Cr{estlip34} diam_{G,0}{S_2}} .
\end{equation}
Here $G \in \mathbb{B}(G_i,\gamma_{G_i})$  and $\mathcal{R}G \in \mathbb{B}(G_j,\gamma_{G_j})$.
\end{lem}
\begin{proof} Due (\ref{compn2}) and the estimates for  $\phi^u$ in the proof of Proposition \ref{trans}  there is  $\Cl{ed}$ such that 
\begin{equation} \label{estprod} \max \{ |D \phi_u|,  |D^2 \phi_u|,|D \phi^{-1}_u| , |D^2 \phi^{-1}_u|\}  \leq \Cr{ed}.\end{equation}
where here we are considering the norms $| \cdot  |_{G,G_i^\star}$ in $E^u_{G}$,  and $| \cdot  |_{\mathcal{R}G,G^\star_j}$ in $E^u_{\mathcal{R}G}$. If $M_x$ is the matrix representation of $D\phi_u(x)$ with respect to the basis $\mathcal{S}_{{G,i}}$ and $\mathcal{S}_{{\mathcal{R}G,j}}$, define
$$J\phi_u(x) = Det \ M_x.$$ It is easy to see that (\ref{estprod}) implies that there exists $\Cl{estlip34}$ satisfying 
\begin{equation}\label{lnjac} \ln \Big| \frac{J\phi^u(x)}{J\phi^u(y)} \Big|\leq \Cr{estlip34} | x-y|_{G}.\end{equation}
Since
$$m_{\mathcal{R}G,j}(\phi_u(S)) = \int_S  J\phi^u   \ dm_{G,i}, $$ 
we have that (\ref{dm}) follows easily from (\ref{lnjac}).
\end{proof}

\begin{prop} \label{comp00} If $\delta > 0$ is small enough there exists $\Cl{dm} > 0$ such that the following holds. We have
that 
$$\mathcal{R}^i\colon B_{\mathcal{H}} (F,\delta,i) \rightarrow B_{\hat{\mathcal{R}}^i (\mathcal{H}_F)}( \mathcal{R}^i F,\delta,0)$$
is a diffeomorphism. Furthermore  for every borelian set $A \subset B_{\mathcal{H}} (F,\delta,i) $ we have
\begin{equation}\label{comp99}
\frac{1}{\Cr{dm}} \frac{m_{\mathcal{H}_F}(A)}{m_{\mathcal{H}_F}(B_{\mathcal{H}} (F,\delta,i))}   \leq  \frac{m_{\mathcal{R}^i (\mathcal{H}_F)}(\mathcal{R}^i(A))}{m_{\mathcal{R}^i (\mathcal{H}_F)}(B_{\hat{\mathcal{R}}^i (\mathcal{H}_F)}( \mathcal{R}^i F,\delta,0))}   \leq    \Cr{dm}\frac{m_{\mathcal{H}_F}(A)}{m_{\mathcal{H}_F}(B_{\mathcal{H}} (F,\delta,i))}.
\end{equation}
\end{prop} 
\begin{proof}  First we will prove by induction on $i$ that 
\begin{equation}\label{dynball2} \mathcal{R}^i\colon B_{\mathcal{H}} (F,\delta,i) \rightarrow B_{\hat{\mathcal{R}}^{i} (\mathcal{H}_F)}( \mathcal{R}^{i} F,\delta,0)\end{equation}
is a diffeomorphism. Indeed, for $i=0$ we have by (\ref{dynball}) that 
 $$B_{\mathcal{H}} (F,\delta,k)= B_{\mathcal{H}_F} (F,\delta,k).$$
Now assume by induction that (\ref{dynball2}) holds for some  $i$. Denote $\mathcal{H}_i:= \hat{\mathcal{R}}^{i} (\mathcal{H}_F)$.
By Proposition \ref{trans} we have that 
$$\mathcal{R}\colon \hat{\mathcal{H}_i}\mapsto \mathcal{R}(\hat{\mathcal{H}_i})$$
is invertible and its inverse $\mathcal{I}$ satisfies 
$$|\mathcal{I}(F_1) - \mathcal{I}(F_2)|_{0}\leq \Cr{expa} |F_1 - F_2|_{0}.$$
so if $$z \in  B_{\hat{\mathcal{R}}^{i+1} (\mathcal{H}_F)}( \mathcal{R}^{i+1} F,\delta,0)$$ we have $z \in  \hat{\mathcal{H}}_{i+1} \subset \mathcal{R}(\hat{\mathcal{H}_i})$ and consequently 
$$|\mathcal{I}(z) - \mathcal{R}^{i}(F)|_{0}\leq \Cr{expa} |z - \mathcal{R}^{i+1}(F)|_{0}\leq  \Cr{expa}  \delta < \delta,$$
so
$$\mathcal{I}( B_{\hat{\mathcal{R}}^{i+1} (\mathcal{H}_F)}( \mathcal{R}^{i+1} F,\delta,0)) \subset B_{\hat{\mathcal{R}}^{i} (\mathcal{H}_F)}( \mathcal{R}^{i} F,\delta,0).$$
By the induction assumption there exists an open set $\mathbb{W} \subset B_{\mathcal{H}} (F,\delta,i)$ such that
$$\mathcal{R}^i (\mathbb{W}) = \mathcal{I}( B_{\hat{\mathcal{R}}^{i+1} (\mathcal{H}_F)}( \mathcal{R}^{i+1} F,\delta,0)).$$
We conclude that 
$$\mathcal{R}^{i+1}\colon  \mathbb{W} \mapsto  B_{\hat{\mathcal{R}}^{i+1} (\mathcal{H}_F)}( \mathcal{R}^{i+1} F,\delta,0)$$
is a diffeomorphism. Of course $\mathbb{W}\subset  B_{\mathcal{H}} (F,\delta,i+1).$  
We claim that  $$\mathbb{W}= B_{\mathcal{H}} (F,\delta,i+1).$$ Notice that 
$$\mathcal{H}_i:= \hat{\mathcal{R}}^{i} (\mathcal{H}_F) \in \mathcal{T}(\mathcal{R}^iG_F, 2\Cr{comp}\delta).$$
By  the proof of Propostition \ref{trans} there exists an open set $W \subset E^u_{\mathcal{R}^{i+1}G_F}$ such that $\mathcal{R}(\hat{ \mathcal{H}_i})$ is the graph of a function
$$\mathcal{H}_{i+1}\colon W \mapsto E^s_{\mathcal{R}^{i+1}G_F} + G_F$$
such that
$$|\mathcal{H}_{i+1}(x)-\mathcal{H}_{i+1}(y)|_{{\mathcal{R}^{i+1}G_F},0} \leq \varepsilon |x-y|_{{\mathcal{R}^{i+1}G_F},0}$$
and
$$\mathbb{B}^u_{\mathcal{R}^{i+1}G_F}(\pi^u(\mathcal{R}^{i+1}F),2\theta_1\Cr{comp}\delta)\subset W.$$
In particular
$$\hat{\mathcal{R}}^{i+1} (\mathcal{H}_F) \in \mathcal{T}(\mathcal{R}^{i+1}G_F, \theta_1  2\Cr{comp}\delta).$$
Note that if 
$$x \in W\setminus  \mathbb{B}^u_{\mathcal{R}^{i+1}G_F}(\pi^u(\mathcal{R}^{i+1}F),2\Cr{comp}\delta)$$
then 
\begin{eqnarray}
& &\ \   |x +  \mathcal{H}_{i+1}(x) - \mathcal{R}^{i+1}F|_{\mathcal{B}_0} \nonumber  \\ &\geq&  \frac{1}{\Cr{comp}}  |x +  \mathcal{H}_{i+1}(x) - \mathcal{R}^{i+1}F|_{\mathcal{R}^{i+1}G_F,0} \nonumber \\
 &\geq& \frac{1}{\Cr{comp}}  |\pi^u_{\mathcal{R}^{i+1}G_F}(x +  \mathcal{H}_{i+1}(x) - \mathcal{R}^{i+1}F)|_{\mathcal{R}^{i+1}G_F,0}\nonumber \\
 &\geq& \frac{1}{\Cr{comp}}  |\pi^u_{\mathcal{R}^{i+1}G_F}(x) - \pi^u_{\mathcal{R}^{i+1}G_F}(\mathcal{R}^{i+1}F)|_{\mathcal{R}^{i+1}G_F,0}\nonumber  \\
 \label{maiordelta} &\geq& \theta_1  2 \delta  > \delta. 
 \end{eqnarray}
So if $y \in B_{\mathcal{H}} (F,\delta,i+1) \subset B_{\mathcal{H}} (F,\delta,i)$ then by induction assumption we have 
$$\mathcal{R}^{i}(y) \in B_{\hat{ \mathcal{H}_i}}( \mathcal{R}^{i} F,\delta,0)$$
and
$$\mathcal{R}^{i+1}(y) \in \mathcal{R}(\hat{ \mathcal{H}_i})$$
and by  (\ref{maiordelta})
 $$ \pi^u_{\mathcal{R}^{i+1}G_F}(\mathcal{R}^{i+1}(y)) \in   \mathbb{B}^u_{\mathcal{R}^{i+1}G_F}(\pi^u(\mathcal{R}^{i+1}F),2\Cr{comp}\delta)$$
 and of course 
 $$\mathcal{R}^{i+1}(y) \in B_{\hat{\mathcal{R}}^{i+1} (\mathcal{H}_F)}( \mathcal{R}^{i+1} F,\delta,0).$$

There is $\tilde{y}\in \mathbb{W}$ such that 
$$\mathcal{R}^{i}(\tilde{y}) \in B_{\hat{ \mathcal{H}_i}}( \mathcal{R}^{i} F,\delta,0)$$
and
$$\mathcal{R}^{i+1}(\tilde{y})=\mathcal{R}^{i+1}(\tilde{y}).$$
Since $\mathcal{R}$ is injective on $\hat{\mathcal{H}}_i$ we conclude that $\mathcal{R}^{i}(\tilde{y})=\mathcal{R}^{i}(y)$. Since by induction assumption $\mathcal{R}^i$ is injective on $B_{\mathcal{H}} (F,\delta,i)$ we conclude that $y=\tilde{y}$. So $\mathbb{W}=B_{\mathcal{H}} (F,\delta,i+1)$ and we conclude that proof that the map in (\ref{dynball2}) is a diffeomorphism. 

It remains to prove  the inequality in the statement of Proposition \ref{comp00}.  Define, for every $i\leq k$
$$S_1^i = \pi^u_{\mathcal{R}^iG_F}(\mathcal{R}^i(A)- \mathcal{R}^i(G_F) ),$$
$$S_2^i = \pi^u_{\mathcal{R}^iG_F}(\mathcal{R}^i(B_{\mathcal{H}}(F,\delta,k))-\mathcal{R}^i(G_F)),$$
Note that 
$$S_1^i \subset S_2^i \subset  \mathbb{B}^u_{\mathcal{R}^{i}G_F}(\pi^u(\mathcal{R}^{i}F),2\Cr{comp}\delta).$$
If 
$$\phi^i_u(x) = \pi^u_{\mathcal{R}^{i+1}G_F}(\mathcal{R}(x+\mathcal{H}_i(x)) - \mathcal{R}^{i+1}G_F)$$
then $\phi^i_u(S_a^i)=S_a^{i+1}$, $a=1, 2$. 
Choosing $j_i$ such that $\mathcal{R}^iG \in \mathbb{B}(G_{j_i},\gamma_{G_{j_i}})$, Lemma \ref{stepmeasure} implies 
\begin{equation}\label{dm3}
\frac{m_{G_F,j_0}(S_1^0)}{m_{G_F,j_0}(S_2^0)} e^{-2\Cr{estlip34}  \sum_{i\leq k} \delta_i}   \leq  \frac{m_{\mathcal{R}^kG_F,j_k}(S_1^k)}{m_{\mathcal{R}^kG_F,j_k}(S_2^k)} \leq    \frac{m_{G_F,j_0}(S_1^0)}{m_{G_F,j_0}(S_2^0)} e^{2\Cr{estlip34} \sum_{i\leq k} \delta_i} .
\end{equation}
with $\delta_i= diam \ S_2^i$.
By (\ref{cinv}) we have 
\begin{equation}\label{dm4}  \delta_i= diam_{\mathcal{R}^iG_F,0} \ S_2^i  \leq 2\Cr{comp} \Cr{expa}^{k-i}\delta\end{equation}
for every $i\leq k$. 
Let $\Cl{po}= 2\Cr{estlip34}  \sum_{k=0}^\infty 2\Cr{comp} \Cr{expa}^{k}$. Then  by (\ref{compk})
\begin{equation}\label{dm5}
\frac{m_{G_F}(S_1^0)}{m_{G_F}(S_2^0)} e^{-\Cr{po}}   \leq  \frac{m_{\mathcal{R}^kG_F}(S_1^k)}{m_{\mathcal{R}^kG_F}(S_2^k)} \leq    \frac{m_{G_F}(S_1^0)}{m_{G_F}(S_2^0)} e^{\Cr{po}} .
\end{equation}
So by (\ref{compkk})
\begin{equation}\label{dm6}
\frac{1}{\Cl{comp55}} \frac{m_{\mathcal{H}_F}(A)}{m_{\mathcal{H}_F}(B_{\mathcal{H}} (F,\delta,k))}  \leq  \frac{m_{\mathcal{H}_k}(\mathcal{R}^k(A))}{m_{\mathcal{H}_k}(B_{\hat{\mathcal{R}}^k (\mathcal{H}_F)}( \mathcal{R}^k F,\delta,0))} \leq  \Cr{comp55}  \frac{m_{\mathcal{H}_F}(A)}{m_{\mathcal{H}_F}(B_{\mathcal{H}} (F,\delta,k))} 
\end{equation}
where $\Cr{comp55}= e^{\Cr{po}} (1+\epsilon)^{4n}$.  Now we need to deal with the fact that $m_{\mathcal{H}}$ may not coincide with  $m_{\mathcal{H}_F}$, since $m_{\mathcal{H}}$ is the $n$-dimensional Hausdorff measure   induced by the norm $|\cdot|_{G,0}$ and $m_{\mathcal{H}_F}$  is is the $n$-dimensional Hausdorff measure induced by the norm $|\cdot|_{G_F,0}$. But by (\ref{comp78}) we have
\begin{equation} \label{dm7} \frac{1}{\Cr{comp}^{2n}}  \leq \frac{m_{\mathcal{H}}(A)}{m_{\mathcal{H}_F}(A)} \leq \Cr{comp}^{2n} .\end{equation} 
From (\ref{dm6}) and (\ref{dm7}) we can easily obtain the inequality in the statement of Proposition \ref{comp00}  with $\Cr{dm}= \Cr{comp}^{2n} \Cr{comp55}$. 
\end{proof}

\begin{proof}[Proof of Proposition \ref{zerohip}]
 It is enough to show that  for every compact subset $K \subset W^s_{\delta_5}(\Omega)\cap\hat{\mathcal{H}}$ we have  $m_{\mathcal{H}}(K)=0.$
Choose $\delta > 0$ small enough such that $\mathcal{H}_F$, as defined in Section \ref{dballs}, is well-defined for every $F \in K$.
Following the notation of Bowen \cite{bowen}, we say that a subset $S \subset K$ is $(k,\delta)$-separated is for every $F_1, F_2 \in S$,  with $F_1\neq F_2$, there exists $i\leq k$ such that 
$$|\mathcal{R}^iF_1 -\mathcal{R}^iF_2|_{0}  > \delta.$$
Fixed $k$ and $\delta$, the family $\mathcal{F}_{k,\delta}$ of $(k,\delta)$-separated subsets of $K\neq \emptyset$ is a non empty family, ordered by the inclusion relation. One can easily see that we can apply Zorn's Lemma to this family to show the it has a maximal element $S_{k,\delta}$. The maximality of $S_{k,\delta}$ implies that
\begin{equation} \label{maxi} K \subset \cup_{F \in S_{k,\delta}} B_{\mathcal{H}}(F,\delta,k).\end{equation}
Since $diam \ B_{\mathcal{H}}(F,\delta,k) \leq 2\Cr{comp} \Cr{expa}^{k}\delta $ we get that 
\begin{equation} \label{limk} \lim_k m_{\mathcal{H}}(\cup_{F \in S_{k,\delta}} B_{\mathcal{H}}(F,\delta,k)) =m_{\mathcal{H}}(K).\end{equation} 
In particular 
\begin{equation} \label{limk1}  \lim_k m_{\mathcal{H}}(K^c\cap \cup_{F \in S_{k,\delta}} B_{\mathcal{H}}(F,\delta,k))=0.\end{equation}  
Note that for every $F \in S_{k,\delta}$ we have 
$$\hat{\mathcal{R}}^k \mathcal{H}_F \in  \mathcal{T}^2_1(\mathcal{R}^kG_F, \frac{\delta}{2 }) \subset \mathcal{T}^2_1(\mathcal{R}^kG_F, \frac{\delta}{2 \Cr{comp}(1+\varepsilon)}) $$
Applying Proposition \ref{measureh} to the family 
$$\hat{\mathcal{H}}_k= \{v+(\hat{\mathcal{R}}^k \mathcal{H}_F)(v)\colon \ v \in \mathbb{B}^u_{\mathcal{R}^kG_F}(\pi^u_{\mathcal{R}^kG_F}(\mathcal{R}^kF-\mathcal{R}^kG_F), \frac{\delta}{2 \Cr{comp}(1+\varepsilon)})\}$$
we conclude that there exists $\Cr{medhip}> 0$, that does not depend on $k$ and $F \in  S_{k,\delta}$ such that 
$$m_{\hat{\mathcal{R}}^k \mathcal{H}_F}(W^s_{\delta_5}(\Omega)^c\cap  \hat{\mathcal{H}}_k \cap \mathbb{B}(\mathcal{R}^kF,\frac{\delta}{2})) > \Cr{medhip}.$$
Note that 
$$W^s_{\delta_5}(\Omega)^c\cap  \hat{\mathcal{H}}_k \cap \mathbb{B}(\mathcal{R}^kF,\frac{\delta}{2}) \subset B_{\hat{\mathcal{R}}^k \mathcal{H}_F}(\mathcal{R}^kF, \frac{\delta}{2},0).$$
So by Lemma \ref{lower} we have 
$$\frac{m_{\hat{\mathcal{R}}^k \mathcal{H}_F}(W^s_{\delta_5}(\Omega)^c\cap B_{\hat{\mathcal{R}}^k \mathcal{H}_F}(\mathcal{R}^kF, \frac{\delta}{2},0) )}{m_{\hat{\mathcal{R}}^k \mathcal{H}_F}(B_{\hat{\mathcal{R}}^k \mathcal{H}_F}(\mathcal{R}^kF, \frac{\delta}{2},0) )} \geq \frac{\Cr{medhip}}{\Cr{lowerballs}(\delta/2)}$$
By Proposition \ref{comp00} there exists a set $A_{F} \subset B_{\mathcal{H}} (F,\delta/2,k)$ such that 
$$\mathcal{R}^k(A_F)= W^s_{\delta_5}(\Omega)^c \cap B_{\hat{\mathcal{R}}^k \mathcal{H}_F}(\mathcal{R}^kF, \frac{\delta}{2},0).$$
Since $W^s_{\delta_5}(\Omega)$ is forward invariant, we have that $W^s_{\delta_5}(\Omega)^c$  is backward invariant and consequently $A_F \subset W^s_{\delta_5}(\Omega)^c\subset K^c$. Moreover
$$\frac{m_{\mathcal{H}}(A_F)}{m_{\mathcal{H}}(B_{ \mathcal{H}}(F, \frac{\delta}{2},k) )} \geq     \frac{\Cr{medhip}}{\Cr{dm} \Cr{lowerballs}(\delta/2)}.$$
Lemma \ref{lower} also implies 
$$\frac{m_{\hat{\mathcal{R}}^k \mathcal{H}_F}(B_{\hat{\mathcal{R}}^k \mathcal{H}_F}(\mathcal{R}^kF, \frac{\delta}{2},0) )}{m_{\hat{\mathcal{R}}^k \mathcal{H}_F}(B_{\hat{\mathcal{R}}^k \mathcal{H}_F}(\mathcal{R}^kF, \delta,0) )} \geq \frac{\Cr{lowerballs2}(\delta/2)}{\Cr{lowerballs}(\delta)},$$ 
so applying Proposition \ref{comp00} again we get 
$$\frac{m_{\mathcal{H}}(B_{ \mathcal{H}}(F, \delta/2,k) )}{m_{\mathcal{H}}(B_{ \mathcal{H}}(F,\delta,k) )} \geq     \frac{\Cr{lowerballs2}(\delta/2)}{\Cr{dm} \Cr{lowerballs}(\delta)} .$$
Because $S_{k,\delta}\in \mathcal{F}_{k,\delta}$ we have that 
$$\{ B_{\mathcal{H}}(F,\delta/2,k)\}_{F \in S_{k,\delta}} $$
is a family of pairwise disjoint dynamical balls. So 

\begin{eqnarray} \label{limk3} m_{\mathcal{H}}(K^c\cap \cup_{F \in S_{k,\delta}} B_{\mathcal{H}}(F,\delta,k))&\geq& m_{\mathcal{H}}( \cup_{F \in S_{k,\delta}} B_{\mathcal{H}}(F,\delta,k)\cap K^c) \nonumber \\
&\geq& m_{\mathcal{H}}(\cup_{F \in S_{k,\delta}}  B_{\mathcal{H}}(F,\delta/2,k)\cap K^c) \nonumber \\
&\geq& \sum_{F \in S_{k,\delta}} m_{\mathcal{H}}( B_{\mathcal{H}}(F,\delta/2,k)\cap K^c) \nonumber \\
&\geq& \sum_{F \in S_{k,\delta}} m_{\mathcal{H}}( A_F) \nonumber \\
&\geq&  \frac{\Cr{medhip}}{\Cr{dm} \Cr{lowerballs}(\delta/2)} \sum_{F \in S_{k,\delta}} m_{\mathcal{H}}( B_{\mathcal{H}}(F,\delta/2,k)) \nonumber \\
&\geq&  \frac{\Cr{medhip} \Cr{lowerballs2}(\delta/2)}{\Cr{dm}^2 \Cr{lowerballs}(\delta/2) \Cr{lowerballs}(\delta)}    \sum_{F \in S_{k,\delta}} m_{\mathcal{H}}( B_{\mathcal{H}}(F,\delta,k)) \nonumber \\ 
&\geq&  \frac{\Cr{medhip} \Cr{lowerballs2}(\delta/2)}{\Cr{dm}^2 \Cr{lowerballs}(\delta/2) \Cr{lowerballs}(\delta)}  m_{\mathcal{H}}( \cup_{F \in S_{k,\delta}} B_{\mathcal{H}}(F,\delta,k))    \nonumber \\ 
&\geq&  \frac{\Cr{medhip} \Cr{lowerballs2}(\delta/2)}{\Cr{dm}^2 \Cr{lowerballs}(\delta/2) \Cr{lowerballs}(\delta)}  m_{\mathcal{H}}(K).
\end{eqnarray} 
It follows easily from (\ref{maxi}),  (\ref{limk1}) and (\ref{limk3}) that  $m_{\mathcal{H}}(K)=0$.
\end{proof}

\begin{cor} \label{a1a2} Let  $\mathcal{H}$ be  a  $C^2$  function satisfying $A1$-$A2$, for some $G \in \Omega$. Then 
$$m_{\mathcal{H}}(W^s_{\delta_5}(\Omega)\cap\hat{\mathcal{H}} )=0.$$
\end{cor} 
\begin{proof} Suppose that $ m_{\mathcal{H}}(W^s_{\delta_5}(\Omega)\cap\hat{\mathcal{H}} )> 0$. Then by the 
 Lebesgue's  density Theorem  there exists a point $F \in W^s_{\delta_5}(\Omega)\cap\hat{\mathcal{H}}$ such that for every open subset $A$ of $\hat{\mathcal{H}}$ such that $ F \in A$ we have 
$$ m_{\mathcal{H}}(W^s_{\delta_5}(\Omega)\cap A )> 0.$$
In particular
$$ m_{\mathcal{H}_F}(W^s_{\delta_5}(\Omega)\cap A )> 0$$
for every open subset $A$ of $\hat{\mathcal{H}}_F$ such that $ F \in A$. Let 
$$\mathcal{H}_i = \hat{\mathcal{R}}^i_{G_F} \mathcal{H}_F.$$
Consequently 
$$ m_{\mathcal{H}_i }(W^s_{\delta_5}(\Omega)\cap A )> 0$$
for every open subset $O$ of $\mathcal{H}_i $ such that $ \mathcal{R}^iF \in O$. By Corollary \ref{cor11}  if $i$ is large enough then  $\mathcal{H}_i$ satisfies $A1$-$A3$. That contradicts  Proposition \ref{zerohip}.
\end{proof}

\section{Proof of the main results.}

\begin{prop} \label{mdense} Let 
$$\mathcal{M}\colon P\rightarrow \mathcal{B}_0$$
be a $C^2$ map, where $P$ is an open subset of a Banach space $\mathcal{B}_2$ and moreover suppose that $D_x\mathcal{M}$ has dense image for every $x \in \mathcal{M}^{-1}W^s_{\delta_5}(\Omega)$. If $\mathcal{R}$ and $\Omega$ satisfy the assumptions of Theorem \ref{main}  then $\mathcal{M}^{-1}W^s_{\delta_5}(\Omega)$ is a $\Gamma^k(\mathcal{B}_2)$-null set, for every $k \in \mathbb{N}\cup \{\infty,\omega_\mathbb{R}\}$.\end{prop}
\begin{proof}
 It is enough to prove that $\mathcal{M}^{-1}W^s_{\delta_5}(\Omega)$ satisfies assumption $(H)$ in Proposition \ref{generic} for every $k \in \mathbb{N}\cup \{\infty,\omega_\mathbb{R}\}$.   There are two cases. \\ \\
\noindent {\it Case I.} If $x \not\in \mathcal{M}^{-1}W^s_{\delta_5}(\Omega)$. Let $\alpha \colon T \rightarrow \mathcal{B}_2$ be the constant family $\alpha(\lambda)=x$, for every $\lambda \in T$.  Since $\mathcal{M}^{-1}W^s_{\delta_5}(\Omega)$ is a closed set,  if $ z\in \mathcal{B}_2$ is small enough we have $m_{\alpha_z}(\mathcal{M}^{-1}W^s_{\delta_5}(\Omega))=0$.\\ 

\noindent {\it Case II.} If $x  \in \mathcal{M}^{-1}W^s_{\delta_5}(\Omega)$.  Then $\mathcal{M} x \in W^s_{\delta_5}(G)$, for some $G \in \Omega$. Since the image of $D_x\mathcal{M}$ is dense, there exists $v_1,  \dots, v_n \in \mathcal{B}_2$ such that 
$$\{\pi^u_{G}\cdot D_x\mathcal{M}\cdot  v_j\}_{j\leq n}$$
is a basis of $E^u_G$ and moreover 
$$ \{ \sum_j \lambda_j D_x\mathcal{M}\cdot  v_j, \ \lambda_j \in \mathbb{R} \} \subset C^u_{\frac{\Cr{3}}{2}  \epsilon, 0}(G).$$ 
Consider the function
$$Q\colon U_1  \times U_2 \rightarrow \mathcal{B}_0$$
where $U_1 \subset \mathcal{B}_2$ is a small neighborhood of $0$  and $U_2\subset \mathbb{R}^n$ is a small  neighborhood of $0$, defined by 
$$Q(z,\lambda_1,\dots,\lambda_n)=  \mathcal{M}(x+z+ \sum_j \lambda_j v_j).$$
Define the $C^2$ map
$$\tilde{Q}\colon U_1  \times U_2  \rightarrow  U_1  \times E^u_G$$ 
as 
$$\tilde{Q}(z,\lambda_1,\dots,\lambda_n)= (z, \pi^u_G\circ Q(z,\lambda_1,\dots,\lambda_n)  - \pi^u_G(G))$$
Note that $D\tilde{Q}_{(0,0,\dots,0)}$ is a continuous linear isomorphism. By the Inverse Function Theorem there is a neighborhood $O=O_1 \times O_2 \subset \mathcal{B}_0\times E^u_G$ of 
$$\tilde{Q}(0,0,\dots,0)= (0,\pi^u_G\circ \mathcal{M}(x)-\pi^u_G(G))$$
and a $C^2$ diffeomorphism  on its image 
$S\colon O \rightarrow S(O),$
where $S(O)$ is a neighborhood of $(0,\dots,0)$ such that $\tilde{Q}\circ S  =Id$. In particular for each $z \in O_1$ the map
$$t \in O_2 \mapsto S(z,t)$$
is a $C^2$ diffeomorphism whose inverse is 
\begin{equation} \label{inversa} (\lambda_1,\dots,\lambda_n) \mapsto H_z(\lambda_1,\dots,\lambda_n) =\pi^u_G\circ Q(z,\lambda_1,\dots,\lambda_n)  - \pi^u_G(G).\end{equation}
For each $z \in O_1$ define
the $C^2$ function
$$\beta_z\colon O_2  \rightarrow E^h_G+G$$
as
$$\beta_z(t)= \pi^h_G\circ Q\circ S(z,t) - \pi^h_G(G)+G.$$
Of course $t + \beta_z(t)= Q\circ S(z,t)$. Reducing $O_2$ to a very small convex  open neighboohood of $\pi^u_G(\mathcal{R}^i(x))- \pi^u_G(G)$ in $E^u_G$ and $O_1$ to  a very small neighborhood of $0$ in $\mathcal{B}_2$  we have that  $\beta_z$  satisfies $A1$-$A2$ for every $z \in O_1$, so by Corollary \ref{a1a2} we conclude that 
\begin{equation}\label{df}  m_{\beta_z}(W^s_{\delta_5}(\Omega))=0.\end{equation}
Choose $\gamma >0$ and $\delta > 0$ such that 
$$\{ z \in \mathcal{B}_2\colon \ |z| < \delta\}\times [-\gamma,\gamma]^n \subset S(O_1\times O_2),$$
and define  
$$\alpha \colon [-1,1]^n  \mapsto \mathcal{B}_2$$
as
\begin{equation} \label{defalpha} \alpha(\lambda_1,\dots,\lambda_n)= x + \sum_i \gamma \lambda_i v_i\end{equation}
Notice that if $|z| < \delta$ 
\begin{eqnarray} \mathcal{M}(\alpha_z(\lambda_1,\dots,\lambda_n))&=& \mathcal{M}( x + z+ \sum_i \gamma \lambda_i v_i) \nonumber \\  &=& Q( z, \gamma \lambda_1, \dots, \gamma \lambda_n)   \nonumber \\ 
&=& \pi^u_G \circ Q( z, \gamma \lambda_1, \dots, \gamma \lambda_n) -  \pi^u_G(G) \nonumber \\
&+&  \pi^h_G \circ Q( z, \gamma \lambda_1, \dots, \gamma \lambda_n) -  \pi^h_G(G) + G \nonumber \\ 
 &=& t+ \beta_z(t),  \end{eqnarray} 
where $t = \pi^u_G\circ Q(z, \gamma \lambda_1,\dots,\gamma \lambda_n)-\pi^u_G(G)=H_z(\gamma \lambda_1,\dots,\gamma \lambda_n)$. Since $t \mapsto S(z,t)$ is a  diffeomorphism whose inverse is $H_z$, by (\ref{df}) we have that 
$$ m_{\alpha_z}(\mathcal{M}^{-1}W^s_{\delta_5}(\Omega))=0.$$
Note that $\alpha \in  \Gamma^k(\mathcal{B})$, for every $k \in \mathbb{N}\cup \{\infty\}$. To finish the proof in the case  $k=\omega_{\mathbb{R}}$   everything we need to do is to  extend $\alpha$ to a  complex affine function 
$$\alpha\colon \mathbb{C}^j \rightarrow \mathcal{B}_\mathbb{C}$$
using (\ref{defalpha}) with $\lambda_i \in \mathbb{C}$. 
\end{proof}

\begin{prop} \label{mdense2}  Let 
$$\mathcal{M}\colon P\rightarrow \mathcal{B}_0$$
be a complex analytic map, where $P$ is an open subset of a complex  Banach space $\mathcal{B}_2$ and moreover suppose that $D_x\mathcal{M}$ has dense image for every $x \in \mathcal{M}^{-1}W^s_{\delta_5}(\Omega)$.  If  $\mathcal{R}$ is a  complex analytic map satisfying the assumptions of Theorem \ref{main2}  then $\mathcal{M}^{-1}W^s_{\delta_5}(\Omega)$ is   $\Gamma^{\omega}(\mathcal{B}_2)$-null set and, considering  $\mathcal{B}_2$ as a real Banach space, all conclusions of Proposition \ref{mdense} holds. 
\end{prop} 
\begin{proof} We can easily replace $[-1,1]$ by $\overline{\mathbb{D}}$  in the proof of  Proposition \ref{mdense} and construct for each $x \in \mathcal{B}_2$ a complex affine function $\alpha \in  \Gamma^\omega(\mathcal{B}_2)$ that satisfies the assumption $(H)$ in Proposition \ref{generic} for $k=\omega$. Considering $\mathcal{B}_2$ as a {\it real } Banach space we can see  $\alpha$ as a real affine transformation
$$\alpha\colon \mathbb{R}^{2n}\rightarrow \mathcal{B}_2$$
defined by
$$\alpha(\lambda_1^a,\lambda_1^b, \dots,\lambda_n^a, \lambda_n^b)= x +z+ \sum_m \gamma (\lambda_m^a+ i \lambda_m^b) v_m.$$
If $\mathcal{B}_{\mathbb{C}}$ is a Banach space complexification  of $\mathcal{B}_2$, every point $w  \in \mathcal{B}_{\mathbb{C}}$ is  a pair $(u,v) \in \mathcal{B}^2_2$, and $\mathcal{B}_{\mathbb{C}}$ is endowed  with the obvious sum and the scalar multiplication
$$(\lambda^a+ i \lambda^b)(u,v)= (\lambda^a u - \lambda^b v, \lambda^a v + \lambda^b u)$$
for every $\lambda^a, \lambda^b \in \mathbb{R}$. We identify $u \in \mathcal{B}_2$ with $(u,0) \in \mathcal{B}_{\mathbb{C}}$. We can extend $\alpha$ to an affine complex map 
$$\alpha_z\colon \mathbb{C}^{2n}\rightarrow  \mathcal{B}_{\mathbb{C}}$$ 
taking $\lambda_m^a,\lambda_m^b \in \mathbb{C}$ and defining 
$$\alpha(\lambda_1^a,\lambda_1^b, \dots,\lambda_n^a, \lambda_n^b)= (x,0) + \sum_m \gamma (\lambda_m^a+ i \lambda_m^b) (v_m,0).$$
 Note that
$$m\{ w  \in \mathbb{R}^{2n}\colon \alpha_z(w) \in  \mathcal{M}^{-1}W^s_{\delta_5}(\Omega), \ and \ |w_i| < \frac{1}{\sqrt{2}}  \}=0.$$
for $z$ small enough and $\alpha_z(w)= z+\alpha(w).$   Define  the affine complex map $$\beta\colon \overline{\mathbb{D}}^{2n}\mapsto \mathcal{B}_\mathbb{C}$$ as $\beta(w)=\alpha(w/\sqrt{2}).$ Then  $\beta \in \Gamma^{k}(\mathcal{B}_2)$ for every $k \in \mathbb{N}\cup\{\infty, \omega_\mathbb{R}\}$, $\beta(0)=x$  and 
$$m\{ w  \in [-1,1]^{2n} \colon \beta_z(w) \in  \mathcal{M}^{-1}W^s_{\delta_5}(\Omega)\}=0,$$
for every $z$ small. So for each $x \in \mathcal{B}_2$ the corresponding $\beta$ satisfies the assumption $(H)$ in Proposition \ref{generic} for every  $k \in \mathbb{N}\cup\{\infty, \omega_\mathbb{R}\}$. This complete the proof. 
\end{proof} 

\begin{proof}[Proof of Theorems \ref{main} and \ref{main2}] Since a countable union of $\Gamma^k(\mathcal{B}_0)$-null sets is a $\Gamma^k(\mathcal{B}_0)$-null set, it is enough to prove that $\mathcal{R}^{-i}W^s_{\delta_5}(\Omega)$ is a $\Gamma^k(\mathcal{B}_0)$-null set for every $i \geq 0$. This follows immediately from Proposition \ref{mdense} and  Proposition \ref{mdense2} taking $\mathcal{M}=\mathcal{R}^i$. 
\end{proof}

\section{The shadow  is  shy.}

Note that if a set $\Theta$ satisfies the assumptions of Proposition \ref{generic} then $\Theta$ is a  locally shy set in the sense of Hunt, Sauer and Yorke \cite{hunt}, and a shy set if $\mathcal{B}$ is separable. Shy sets are the same as  Haar null sets in abelian polish groups as defined by Christensen \cite{haarnull} in the case of separable Banach spaces. Consequently if $\mathcal{B}$ is separable and  under the assumptions of either Proposition \ref{mdense}, Proposition \ref{mdense2}, Theorem \ref{main} or Theorem \ref{main2} the sets under consideration are also shy sets (Haar null sets). 

\section{Curves transversal to a  global stable lamination.}

The following results  are  useful to prove that  {\it specific } families $\gamma$ satisfy $$m_\gamma(\cup_{i\geq 0}\mathcal{R}^{-i}W^s_{\delta_5}(\Omega))=0.$$

\begin{prop}\label{transtable} Let $\mathcal{R}$, $\Omega$, $\mathcal{B}_0$ be as in  Theorem \ref{main} and 
$$\mathcal{M}_i\colon P_i \rightarrow \mathcal{B}_0,$$
where   $i \in \Lambda \subset \mathbb{N}$ and $P_i$ are  open subsets of a Banach space $\mathcal{B}_2$, be $C^2$ maps.  Let $\gamma \in \Gamma^2(\mathcal{B}_2)$ and  
$K_i$, $i \in \Lambda$,  be the subset of parameters $\lambda \in (-1,1)^n$  such that 
\begin{itemize} 
\item[i.] We have that $\gamma(\lambda) \in P_i$,
\item[ii.] We have $\mathcal{M}_i\circ \gamma(\lambda) \in W^s_{\delta_5}(G)$ for some $G \in \Omega$,
\item[iii.]  The derivative $D_\lambda ( \mathcal{M}_i\circ \gamma)$   is injective   and 
\item[iv.]  The image of $D_\lambda (\mathcal{M}_i\circ \gamma)$ is contained in $C^u_{\frac{\Cr{3}}{2}  \epsilon, 0}(G)$. 
\end{itemize}
Then $$m(\cup_{i\in \Lambda} K_i)=0.$$ The same holds if we replace the assumptions of Theorem \ref{main}   by the assumptions of  Theorem \ref{main2}, $(-1,1)^n$ by $\mathbb{D}^n$ and $\gamma \in \Gamma^\omega(\mathcal{B}_2)$.
\end{prop}
\begin{proof} Suppose that $m(K_i) > 0$ for some $i \in \Lambda$.  Let $\lambda_0 \in K_i$ be a Lesbegue  density point of $K_i$. Due iii.  and iv.  the derivative   $D_{\lambda_0} ( \pi_G^u \circ \mathcal{M}_i\circ \gamma)$ is injective so we   can use the same method in the proof of Proposition \ref{mdense} to show that there exists a  small neighboorhood $O$ of $\lambda_0$ in $(-1,1)^n$ such that $\mathcal{M}_i\circ \gamma (O)$ is the graph of a $C^2$ function $\mathcal{H}\colon U \mapsto E^h_G+G$, where $U$ is an open subset of $E^u_G$,  and there is a $C^2$ diffeomorphism $S\colon U \mapsto O$ such that $F+\mathcal{H}(F)= \mathcal{M}_i\circ \gamma (S(F))$. Because the image of $D_\lambda (\mathcal{M}_i\circ \gamma)$ is contained in $C^u_{\frac{\Cr{3}}{2}  \epsilon, 0}(G)$ we can reduce $O$ such that $\mathcal{H}$  satisfies conditions $A1$-$A2$ in Corollary  \ref{a1a2}. Since $\lambda_0$ is  a Lesbegue  density point of $K_i$ we conclude that $m_{\mathcal{H}}(W^s_{\delta_5}(G))> 0$. That  contradicts the conclusion of  Corollary \ref{a1a2}. A similar proof also works in the second case. 
\end{proof}

An important case is when we are under the assumptions of Theorem \ref{main}, $\Omega$ is hyperbolic and its stable  directions are the horizontal directions $E^h$. Suppose that 
$$\mathcal{M}\colon P\rightarrow \mathcal{B}_0$$
is  a $C^2$ map, where $P$ is an open subset of a Banach space $\mathcal{B}_2$. For every $G \in \Omega$ define the set 
$$W^{st}(\mathcal{M},G) =\{F \in P \colon \lim_{i} | \mathcal{R}^i \circ \mathcal{M}(F)- \mathcal{R}^i G|_0=0    \}$$
and the {\it global stable  lamination }
$$W^{st}(\mathcal{M},\Omega)= \cup_{G\in \Omega} W^{st}(\mathcal{M},G).$$
Moreover suppose that $D_F \mathcal{R}^i\circ \mathcal{M}$, $i \in \mathbb{N}$,  is injective and it has  dense image for every 
$$F \in W^{st}(\mathcal{M},\Omega).$$
For each $F \in W^{st}(\mathcal{M},G)$ define the subspace 
$$E^h_F= \{ v \in \mathcal{B}_2 \colon \sup_i |D_F (\mathcal{R}^i \circ \mathcal{M})\cdot v|_0 < \infty \}.$$
Note that 
\begin{itemize}
\item[i.] If $v \notin E^h_F$  then for every $\epsilon' > 0$  there exists $i_{(F,v)}$ such that  
$$D_F(\mathcal{R}^i \circ \mathcal{M})\cdot v \in C^u_{\epsilon',0}(\mathcal{R}^iG)$$
for every $i > i_{(F,v)}$,
\item[ii.] The subspace $E^h_F$ is closed,
\item[iii.] $codim \ E^h_F= n.$
\item[iv.] There exists $C\geq 0$ and $\lambda \in (0,1)$ such that 
$$|D_F(\mathcal{R}^i \circ \mathcal{M})\cdot v|_0\leq C\lambda^i |v|_0$$
for every $v \in E^h_F$. 
\end{itemize}
Indeed  (i) and (iv) follows from the hyperbolicity of $\Omega$. Property (ii) follows from (i), since $D_{\mathcal{R}^i \circ \mathcal{M}(F)} \mathcal{R}$ maps  $C^u_{\epsilon, 0}(\mathcal{R}^iG)\setminus \{0\}$  strictly inside  $$C^u_{\epsilon, 0}(\mathcal{R}^{i+1} G)\setminus\{0\}$$  for every large $i$ and it expands vectors in that cone. Since $D_F \mathcal{R}^i\circ \mathcal{M}$, $i \in \mathbb{N}$  is injective and it has  dense image, it is easy to prove (iii). 

\begin{cor} Under the above conditions we have   $m_{\gamma}(W^{st}(\mathcal{M},\Omega))=0$ for every $\gamma \in \Gamma^2_n(\mathcal{B}_0)$ such that $\gamma \pitchfork W^{st}(\mathcal{M}, \Omega)$, that is, if $\gamma(\lambda) \in  W^{st}(\mathcal{M}, \Omega)$ then $Im \ D_\lambda \gamma \pitchfork E^h_{\gamma(\lambda)}$. The same conclusion holds if we are under the conditions of Theorem \ref{main2} and $\gamma \in \Gamma^\omega_n(\mathcal{B}_0)$. \end{cor}
\begin{proof} If $\gamma(\lambda_0)=F \in W^{st}(\mathcal{M},G)$ for some $G \in \Omega$  we have that for $i$ large enough $\mathcal{R}^i\circ \mathcal{M}(F) \in W_{\delta_5}^s(\mathcal{R}^i(G))$, 
$$Im \ D_{\lambda_0}(\mathcal{R}^i  \circ \mathcal{M} \circ \gamma) \subset C^u_{\frac{\Cr{3}}{2}  \epsilon, 0}(\mathcal{R}^i(G)),$$
and moreover 
$$D_{\lambda_0}(\pi^u_{\mathcal{R}^iG} \circ \mathcal{R}^i  \circ \mathcal{M} \circ \gamma)$$
is injective.  Take $\mathcal{M}_i =\mathcal{R}^i\circ \mathcal{M}$ and apply Proposition \ref{transtable}. 
\end{proof}

\bibliographystyle{abbrv}

\end{document}